\documentclass[10pt]{article}

\usepackage{amsmath}
\usepackage{graphicx}
\usepackage{enumerate}
\usepackage{natbib}
\usepackage{url} 
\RequirePackage[colorlinks,citecolor=blue,urlcolor=blue]{hyperref}
\usepackage{latexsym,amssymb,amsopn,amscd,
amsbsy,marginnote,cprotect}
\usepackage{mathrsfs}
\usepackage{subfig}
\usepackage[usenames,dvipsnames]{color}
\usepackage[normalem]{ulem}
\usepackage{float}
\usepackage{epsfig}
\usepackage{relsize}
\usepackage{xcolor}
\usepackage{bm}    
\usepackage{bbm}                   
\usepackage[english]{inputenc}  
\usepackage{amsthm} 

\newcommand{\N}{\ensuremath{\mathbb{N}}}
\newcommand{\R}{\ensuremath{\mathbb{R}}}

\newcommand{\Prob}{\ensuremath{\mathbb{P}}}
\newcommand{\Esp}{\ensuremath{\mathbb{E}}}
\newcommand{\A}{\ensuremath{\forall}}

\newcommand{\B}{\mathscr{B}}
\newcommand{\F}{\cl{F}}

\newcommand{\Ind}{\mathlarger{\mathbf{1}}}
\newcommand{\captionstring}[1]{\noexpand\noexpand\noexpand\string\string#1}

\newcommand{\cl}[1]{\mathcal{#1}}

\newcommand{\msf}[1]{\mathsf{#1}}
\newcommand{\mbf}[1]{\mathbf{#1}}

\newcommand{\mrm}[1]{\mathrm{#1}}

\def\mi{\,\middle|\,}
\def\l{\left}
\def\r{\right}

\def\deq{\stackrel{d}{=}}

\def\dto{\stackrel{d}{\to}}

\def\wto{\stackrel{w}{\to}}
\def\dwto{\stackrel{dw}{\to}}
\def\L2to{\stackrel{\cl{L}_2}{\to}}

\def\Lpwto{\stackrel{\cl{L}_p w}{\to}}
\def\Pwto{\stackrel{\Prob w}{\to}}
\def\iid{\stackrel{\mbox{\scriptsize{iid}}}{\sim}}

\def\d{{\mbf{d}}}
\def\k{{\mbf{k}}}
\def\u{{\mbf{u}}}
\def\v{{\mbf{v}}}
\def\w{{\mbf{w}}}
\def\x{{\mbf{x}}}
\def\y{{\mbf{y}}}
\def\z{{\mbf{z}}}

\def\n{{\mbf{n}}}
\def\m{{\mbf{m}}}
\def\k{{\mbf{k}}}
\def\p{{\mbf{p}}}

\def\W{{\mbf{W}}}
\def\Y{{\mbf{Y}}}

\def\K{{\mbf{K}}}
\def\U{{\mbf{U}}}
\def\V{{\mbf{V}}}
\def\X{{\mbf{X}}}

\def\D{{\mbf{D}}}

\def\bmu{{\bm{\mu}}}

\def\bnu{{\bm{\nu}}}

\def\bxi{{\bm{\xi}}}

\def\bXi{{\bm{\Xi}}}
\def\bPhi{{\bm{\Phi}}}

\def\bPi{{\bm{\Pi}}}
\def\bgamma{{\bm{\gamma}}}
\def\bbp{\mathbbm{p}}

\def\tw{{\mbf{\tilde{w}}}}

\def\tW{{\mbf{\tilde{W}}}}

\def\mathL{\cl{L}}

\def\Be{{\msf{Be}}}
\def\Un{{\msf{Unif}}}

\def\SB{{\msf{SB}}}

\def\Ga{{\msf{Ga}}}

\def\Var{{\msf{Var}}}
\def\Cov{{\msf{Cov}}}
\def\Corr{{\msf{Corr}}}
\def\2F1{{{}_2{F}_1}}

\theoremstyle{plain}
\newtheorem{theo}{Theorem}[section]

\newtheorem{lem}[theo]{Lemma}
\newtheorem{cor}[theo]{Corollary}

\theoremstyle{remark}
\newtheorem{rem}[theo]{Remark}

\title{\bf Stick-breaking processes with exchangeable length variables}

\author{
  Mar\'ia F. Gil--Leyva\\
  IIMAS, Universidad Nacional Aut\'onoma de M\'exico\\
  CDMX, M\'exico\\
  \texttt{marifer@sigma.iimas.unam.mx} \\
  \and
  Rams\'es H. Mena\\
  IIMAS, Universidad Nacional Aut\'onoma de M\'exico\\
  CDMX, M\'exico\\
  \texttt{ramses@sigma.iimas.unam.mx} \\
}

\date{}

\def\spacingset#1{\renewcommand{\baselinestretch}%
{#1}\small\normalsize} \spacingset{1}

\begin{document}
\maketitle

\begin{abstract}
Our object of study is the general class of stick-breaking processes with exchangeable length variables. These generalize well-known Bayesian non-parametric priors in an unexplored direction. We give conditions  to assure the respective species sampling process is proper and the corresponding prior has full support. {For a rich sub-class we explain how, by tuning a single $[0,1]$-valued parameter, the stochastic ordering of the weights can be modulated, and Dirichlet and Geometric priors can be recovered.} A general formula for the distribution of the latent allocation variables is derived and an MCMC algorithm is proposed for density estimation purposes. 
\end{abstract}

\noindent%
\vfill

\newpage
\spacingset{1.1}

\section{Introduction}
Bayesian non-parametric priors have gained interest mainly due to their great flexibility to adjust to complex data sets, while remaining mathematically tractable. The Dirichlet process \citep{F73} stands as the canonical and most popular non-parametric prior in the literature. {This random probability measure enjoys the property of having exchangeable increments with respect to a suitable measure, and when defined over a Polish space, it has full support with respect to the weak topology}. These two  characteristics influence greatly the fact that the Dirichlet process is such a pliable model. 

Searching for competitive alternatives to the canonical model, different constructions of random probability measures {over Polish spaces} have been developed. Some of the most notable are through the normalization of homogeneous completely random measures \citep{K75,Regazzini03,JLP09,HHMW10}, through  the prediction rule of exchangeable partitions \citep{BM73,P06}, by means of the stick-breaking decomposition of a sequence of weights \citep{S94,IJ01,FLP12}, and most recently by virtue of latent random subsets of the natural numbers \citep{W07,FMW10,DMMP20}. All of these, as well as other popular models, fall into the class of species sampling processes \citep{P06,JangLeeLee2010}, which are random probability measures whose atomic decomposition specializes to the form,
\begin{equation}\label{eq:ssp}
\bmu = \sum_{j \geq 1}\w_j\delta_{\bxi_j} + \l(1-\sum_{j \geq 1}\w_j\r) \mu_0,
\end{equation}
for some collection of non-negative random variables $\W = (\w_j)_{j \geq 1}$ satisfying $\sum_{j \geq 1} \w_j \leq 1$ almost surely, independently of the atoms, $(\bxi_j)_{j \geq 1}$,  that are independent and identically distributed (i.i.d.) from the diffuse probability measure $\mu_0$, called base measure of $\bmu$. Indeed, $\bmu$ as in \eqref{eq:ssp} has $\mu_0$-exchangeable increments \citep{K17}, and under some conditions on the weights, $(\w_j)_{j \geq 1}$, as characterized by \cite{BO14}, the corresponding prior has full support with respect to the weak topology.

In contrast to other constructions, the stick-breaking decomposition is exhaustive in the class of species sampling processes, in the sense that the weights of every random probability measure in this class have a stick-breaking decomposition. For completeness, a proof of this {last statement}  can be found in Theorem \ref{theo:sb} in the appendix.  Namely, we can decompose
\begin{equation}\label{eq:sb}
\w_1 = \v_1, \quad \w_j = \v_j\prod_{i=1}^{j-1}(1-\v_i), \quad j \geq 2,
\end{equation}
for some {sequence} $\V = (\v_i)_{i \geq 1}$, with elements taking values in $[0,1]$. {Hereinafter we write $(\w_j)_{j \geq 1} = \SB[(\v_i)_{i \geq 1}]$ (or $\W = \SB[\V]$) whenever \eqref{eq:sb} holds, and refer to the elements of $\V$ as length variables}. Most efforts have concentrated in the case where these are mutually independent \citep{S94,P96b,IJ01,RD11,RQ15}, while there are only a handful of examples of stick-breaking process with explicitly dependent length variables \citep{FMW10,FLP12,Favaro2016,GMN20}. So far, the dependent case has remained somehow elusive due to the mathematical hurdles to overcome. Our proposal here represents, to some extent, a first general treatment of stick-breaking processes with dependent length variables. Explicitly, we will be focusing on the class with exchangeable length variables. Such stick-breaking processes not only constitute a rich class of priors that maintain mathematical tractability, but also delve into unexplored territory that unifies current theory of stick-breaking processes with independent and dependent length variables. 

The outline of the paper is as follows. Section \ref{sec:exch_ssp} presents an overview of exchangeable sequences driven by a species sampling processes. In Section~\ref{sec:SBE} we analyze the  general case of stick-breaking processes with exchangeable length variables. Section \ref{sec:SBESSP} addresses the case where the length variables themselves are driven by another species sampling process. Even this subclass is substantially wide as, in particular, it generalizes Dirichlet and Geometric processes \citep{FMW10}. An important result characterizing the ordering of the respective weights is given. For illustration purposes, in Section \ref{sec:SBEex} we consider the case where the length variables are driven by a Dirichlet process and other interesting species sampling processes such as the Pitman-Yor process. Finally, in Section \ref{sec:Illust}, we implement and evaluate our models to estimate the density of univariate and bivariate simulated data. The proofs of main results as well as the MCMC algorithm are deferred to the appendix.

\section{Preliminaries}\label{sec:exch_ssp}

One of the most influential theoretical results for Bayesian statistics is the representation theorem for exchangeable sequences, first proved by \cite{deFinetti31} and {later} generalized  by \cite{HS55}. This result states that a sequence, $\X = (\x_i)_{i \geq 1}$, whose elements take values in a Polish space, $S$, with Borel $\sigma$-algebra, $\B_S$, is exchangeable if and only if there exist a random probability measure, $\bmu$, over $(S,\B_S)$, such that for every $n \geq 1$ and $B_1,\ldots,B_n \in \B_S$, $\Prob[\x_1\in B_1,\ldots,\x_n \in B_n\mid \bmu] = \prod_{i=1}^{n}\bmu(B_i).$ Hence, elements in $\X$ are conditionally i.i.d. given $\bmu$, denoted by $\{\x_i\mid \bmu \iid \bmu ; \, i \geq 1\}$. The random measure $\bmu$ is called the directing random measure of $\X$, it is unique almost surely and given by the almost sure limit of the empirical distributions $\bmu = \lim_{n \to \infty} n^{-1}\sum_{i=1}^n \delta_{\x_i}$. 

A random variable that encloses important information about an exchangeable sequence is the random partition of $\{1,\ldots,n\}$, here denoted by $\bPi(\x_{1:n})$, generated by the random equivalence relation $i \bm{\sim} j$ if and only if $\x_i = \x_j$. Using the terminology of \cite{P95}, $\bPi(\x_{1:n})$ is exchangeable, for every $n \geq 1$, in the sense that for any partition $A = \{A_1,\ldots,A_k\}$ of $\{1,\ldots,n\}$, $\Prob[\bPi(\x_{1:n}) = A] = \pi(|A_1|,\ldots,|A_k|)$, for some symmetric function, $\pi:\bigcup_{k \in \N}\N^{k} \to [0,1]$, where $|A_i|$ stands for the cardinality of $A_i$. The function $\pi$ is called exchangeable partition probability function (EPPF). Another key aspect is that the collection $(\bPi(\x_{1:n}))_{n \geq 1}$ is consistent \citep{P95,P06} meaning that the restriction of $\bPi(\x_{1:n})$ to $\{1,\ldots,m\}$ equals $\bPi(\x_{1:m})$ almost surely, for every $n > m$. This translates to the well known addition rule of the EPPF
\begin{equation}\label{eq:add_rule}
\pi(n_1,\ldots,n_k) = \pi(n_1,\ldots,n_k,1) + \sum_{j=1}^k \pi(n_1,\ldots,n_{j-1},n_j+1,n_{j+1},\ldots,n_k).
\end{equation}
In fact, every symmetric function $\pi:\bigcup_{k \in \N}\N^{k} \to [0,1]$, with $\pi(1) = 1$ and that satisfies \eqref{eq:add_rule}, defines the EPPF of a consistent family of exchangeable partitions \citep[][]{P95}.

In particular, if the directing random measure, $\bmu$, is a species sampling process, the base measure, $\mu_0$, together with the EPPF, $\pi$, characterize completely the {laws} of the $\bmu$ and $(\x_i)_{i \geq 1}$. To be precise $\x_1 \sim \mu_0 = \Esp[\bmu]$, and for every $n \geq 1$, the prediction rule is
\begin{equation}\label{eq:pred_rule}
\Prob[\x_{n+1} \in \cdot\mid \x_1,\ldots,\x_{n}] = \sum_{j=1}^{\K_n}\frac{\pi\l(\n^{(j)}\r)}{\pi(\n)}\delta_{\x^*_j} + \frac{\pi\l(\n^{(\K_n+1)}\r)}{\pi(\n)}\mu_0,
\end{equation}
where $\x^*_1,\ldots,\x^{*}_{\K_n}$ are the $\K_n$ distinct values in $\{\x_1,\ldots,\x_n\}$, $\n_{j} = |\{i\leq n:\x_i = \x^{*}_j\}|$, for $j \leq \K_n$, $\n = (\n_1,\ldots\n_{\K_n})$,  $\n^{(j)} = (\n_1,\ldots \n_{j-1},\n_j+1,\n_{j+1}, \ldots,\n_{\K_n})$ and  $\n^{(\K_n+1)} = (\n_1,\ldots,\n_{\K_n},1)$. Equivalently, for every $n \geq 1$ and every measurable function $f:S^{n} \to \R$,
%\begin{equation}
\begin{align}
&\Esp\l[f(\x_1,\ldots,\x_n)\r]\label{eq:means_samples}\\
& = \sum_{\{A_1,\ldots,A_k\}}\l\{\int f(x_{l_1},\ldots,x_{l_n})\prod_{j=1}^{k}\prod_{r \in A_j}\Ind_{\{l_r = j\}}\,\mu_0(dx_1)\ldots\mu_0(dx_k)\r\}\pi(|A_1|,\ldots,|A_k|),\nonumber
\end{align}
%\end{equation}
whenever the integrals in the right side exist, and where the sum ranges over all partitions of $\{1,\ldots,n\}$. Equality \eqref{eq:means_samples} is a generalization of a result by \cite{Y84},  and its proof can be found in Theorem \ref{theo:means_samples} in the Appendix \ref{sec:exch_ssp_supp}. A very important quantity of species  sampling process is the tie probability  of $\bmu$,
\[
\rho = \pi(2) = \Prob[\x_1 = \x_2] =  \Esp\l[\Prob[\x_1 = \x_2\mid \bmu]\r] = \sum_{j \geq 1}\Esp\l[\l(\w_j\r)^2\r].
\]
For the case $n = 2$, and using the exchangeability of the sequence, \eqref{eq:means_samples} simplifies to
\[
\Esp[f(\x_i,\x_j)] = \Esp\l[f(\x_1,\x_2)\r] = \rho \int f(x,x) \mu_0(dx) +(1-\rho)\int f(x_1,x_2) \mu_0(dx_1)\mu_0(dx_2),
\]
for every $i \neq j$, whenever $\int f(x,x) \mu_0(dx)$ and $\int f(x_1,x_2) \mu_0(dx_1)\mu_0(dx_2)$ exist. Another quantity that can be written in terms of the tie probability, is the conditional probability,
\[
\Prob[\x_j \in \cdot\mid\x_i] = \frac{\pi(2)}{\pi(1)}\delta_{\x_i} + \frac{\pi(1,1)}{\pi(1)}\mu_0 = \rho\,\delta_{\x_i} + (1-\rho)\mu_0,
\]
for every $i \neq j$. The last equation follows easily from \eqref{eq:pred_rule}, using the fact that $(\x_i,\x_j)$ is equal in distribution to $(\x_1,\x_2)$. Note that when $\rho \approx 0$, $\Prob[\x_j \in \cdot\mid\x_i] \approx \mu_0$, in the opposite case when $\rho \approx 1$, $\Prob[\x_j \in \cdot\mid\x_i] \approx \delta_{\x_i}$. Indeed, as the tie probability approaches zero, $(\x_i)_{i \geq 1}$ converges in distribution to an i.i.d. sequence, and as $\rho \to 1$, $(\x_i)_{i \geq 1}$ converges in distribution to a sequence of identical random variables. As we will see this assertion is essential to prove important convergence properties of stick-breaking process with exchangeable length variables. An account of exchangeable sequences driven by species sampling process is provided in the Appendix \ref{sec:exch_ssp_supp} along with the proof of the aforementioned properties.

\section{Stick-breaking processes with exchangeable length variables}\label{sec:SBE}

Hereinafter we focus on species sampling processes, $\bmu$, as in \eqref{eq:ssp}, defined over a measurable Polish space, $(S,\B_S)$, with collection of weights $\W = \SB[\V]$, where $\V = (\v_i)_{i \geq 1}$ is an exchangeable sequence. In other words, we analyze species sampling process whose weights' distribution remains invariant under permutations of the length variables. Any random probability measure  of this kind will be referred as an exchangeable stick-breaking process (ESB). {The first examples of species sampling processes in this class are well-known models in the literature. In fact, if the length variables are identical, $\v_i = \v$, for $i \geq 1$, and some $\v \sim \nu_0$, where $\nu_0$ is a diffuse probability measure over $\l([0,1],\B_{[0,1]}\r)$, then $(\v,\v,\ldots)$ is exchangeable and driven by $\bnu = \delta_{\v}$. In this case, where the length variables are fully dependent, the decreasingly ordered Geometric weights, $\w_j = \v(1-\v)^{j-1}$, for $j \geq 1$, are recovered, so that the corresponding ESB becomes a Geometric process \citep{FMW10}. In terms of the dependence between length variables, at the other end of the spectrum, we find i.i.d length variables, $\{\v_i \iid \nu_0;\, i \geq 1\}$. This sequence, $(\v_i)_{i \geq 1}$, is trivially exchangeable and driven by the deterministic probability measure $\bnu = \nu_0$, clearly the respective ESB recovers a stick-breaking process featuring i.i.d. length variables \citep[see for example][]{S94,IJ01,P96b}. In particular, if $\nu_0$ stands for a Beta distribution with parameters $(1,\theta)$, denoted by $\Be(1,\theta)$, so that $\{\v_i \iid \Be(1,\theta);\, i \geq 1\}$, the ESB corresponds to a Dirichlet process with total mass parameter $\theta$. 

From a Bayesian perspective, especially if the distribution of the species sampling process is to be regarded as a mixing prior, one of the first properties one should analyze is whether
the weights sum up to one. In this case the almost surely discrete species sampling process is termed proper. Another property of interest is whether a species sampling process has full support. This property assures that if the support of the base measure, $\mu_0$, is $S$, then the weak topological support of the distribution of the species sampling process is the set of all probability measures over $(S,\B_S)$ \citep[see][]{BO14}. Dirichlet and Geometric processes are both proper and have full support under minor conditions of $\nu_0$, the following theorem generalizes this result to the complete class of ESBs.
}

\begin{theo}\label{theo:exch_SB}
Let $\bmu$ be an ESB, with exchangeable length variables, $(\v_i)_{i\geq 1}$, driven by the random probability measure $\bnu$. Let us denote $\nu_0 = \Esp[\bnu]$. 
\begin{itemize}
\item[i)] If there exists $\varepsilon > 0$ such that $(0,\varepsilon)$ is contained in the support of $\nu_0$, $\bmu$ has full support.
\item[ii)] $\bmu$ is proper if and only if $\bnu(\{0\})< 1$ almost surely.
\end{itemize}
\end{theo}

In the context of the above theorem, if $0 = \nu_0(\{0\}) = \Esp[\bnu(\{0\})]$, we have that $\bnu(\{0\}) = 0$ almost surely. Hence, a sufficient condition to ensure $\sum_{j \geq 1}\w_j = 1$, is that $0$ is not an atom of $\nu_0$, that is to say, $\Prob[\v_i = 0] = 0$. For example, if $\nu_0 = \Be(a,b)$ for some $a,b > 0$, so that marginally $\v_i \sim \Be(a,b)$, then $0$ is not an atom of $\nu_0$. Furthermore, the support of a $\Be(a,b)$ distribution is $[0,1]$, which means that the conditions given in (i) and (ii) of Theorem \ref{theo:exch_SB} are satisfied and we have the following corollary.

\begin{cor}\label{cor:exch_SB_beta}
Let $\bmu$ be an ESB, with exchangeable length variables, $(\v_i)_{i\geq 1}$, such that marginally $\v_i \sim \Be(a,b)$. Then, $\bmu$ is proper and it has full support. 
\end{cor}

{
Notice that Geometric processes with a Beta distributed length variable and stick-breaking processes featuring i.i.d. Beta length variables, including Dirichlet processes, are all particular cases of the species sampling processes in Corollary \ref{cor:exch_SB_beta}. Our next result explains how Geometric and Dirichlet processes can be recover as limits of other ESBs.

\begin{theo}\label{theo:exch_SB_limit_0}
Consider some diffuse probability measures $\mu_0,\mu^{(1)}_0,\mu^{(2)}_0,\ldots$, over $(S,\B_S)$, such that as $n \to \infty$, $\mu^{(n)}_0$ converges weakly in distribution to $\mu_0$. For each $n \geq 1$, let $\bnu^{(n)}$ be a random probability measure over $\l([0,1],\B_{[0,1]}\r)$, with $\bnu^{(n)}(\{0\}) < 1$ almost surely, and let $\bmu^{(n)}$ be an ESB with base measure $\mu_0^{(n)}$ and exchangeable length variables $\V^{(n)} = \l(\v^{(n)}_i\r)_{i \geq 1}$ with directing random measure $\bnu^{(n)}$. Denote by $\W^{(n)} = \SB\l[\V^{(n)}\r]$ the weights of $\bmu^{(n)}$.
\begin{itemize}
\item[i)] If $\bnu^{(n)}$ converges weakly in distribution to a deterministic probability measure $\nu_0$ with $\nu_0 \neq \delta_{0}$, then $\bmu^{(n)}$ converges weakly in distribution to a stick-breaking process, $\bmu$, with base measure $\mu_0$ and featuring independent length variables $\{\v_i \iid \nu_0;\, i \geq 1\}$, as $n \to \infty$. In particular if $\nu_0 = \Be(1,\theta)$, the limit, $\bmu$, is a Dirichlet process with total mass parameter $\theta$, and $\W^{(n)}$ converges in distribution to the size-biased permuted weights of $\bmu$.
\item[ii)] If $\bnu^{(n)}$ converges weakly in distribution to $\delta_{\v}$ for some $[0,1]$-valued random variable $\v \sim \nu_0$, where $\nu_0(\{0\}) = 0$. Then, as $n \to \infty$, $\bmu^{(n)}$ converges weakly in distribution to a Geometric process, $\bmu$, with base measure $\mu_0$ and length variable $\v \sim \nu_0$, and $\W^{(n)}$ converges in distribution to the decreasingly ordered weights of $\bmu$.
\end{itemize}
\end{theo}
}

Some important quantities that contain relevant information of the model, are the expectations $\Esp\l[\prod_{j=1}^k\v_j^{a_j}(1-\v_j)^{b_j}\r]$, for non-negative integers $(a_j,b_j)_{j=1}^k$. {For instance, as shown by \cite{P96b,P96}, for a proper species sampling process $\bmu$, with weights sequence $(\w_j)_{j \geq 1}$, taking the form \eqref{eq:sb}, its EPPF is given by
\begin{equation}\label{eq:eppf_sum}
\pi(n_1,\ldots,n_k) = \sum_{(j_1,\ldots,j_k)}\Esp\l[\prod_{i=1}^k\w_{j_i}^{n_i}\r] = \sum_{(j_1,\ldots,j_k)}\Esp\l[\prod_{i=1}^m \v_i^{a_i}(1-\v_i)^{b_i}\r],
\end{equation}
where the sum ranges over all $k$-tuples of distinct positive integers $(j_1,\ldots,j_k)$, and $m$ and $(a_i,b_i)_{i=1}^m$ are functions of $(j_1,\ldots,j_k)$, given by $m = \max\{j_1,\ldots,j_k\}$, $a_i = \sum_{l=1}^{m}n_l\Ind_{\{j_i = l\}}$, and $b_i = \sum_{r > i}a_r$.} \cite{P96b} also proved that
\begin{equation}\label{eq:pepf}
\pi'(n_1,\ldots,n_k) = \Esp\l[\prod_{j=1}^k\w_j^{n_j-1}\prod_{j=1}^{k-1}\l(1-\sum_{i=1}^j \w_i\r)\r] = \Esp\l[\prod_{j=1}^k\v_j^{n_j-1}(1-\v_j)^{\sum_{i>j}n_i}\r]
\end{equation}
is a symmetric function of $(n_1,\ldots,n_k)$ if and only if $(\w_j)_{j \geq 1}$ is invariant under size-biased permutations. In which case $\pi = \pi'$ is precisely the EPPF corresponding to the model. {Evidently, if the distribution of the size-biased permuted weights is not available, computing the EPPF in closed form can be difficult due to the infinite unordered sum in \eqref{eq:eppf_sum}. For a handful of stick-breaking processes, despite the lack of invariance under size-biased permutations, \eqref{eq:eppf_sum} can be simplified \citep[e.g.][]{MW12,RQ15}, however, for most general cases the equation in question cannot be reduced to a more tractable expression.} A way to mitigate this hurdle, is to undertake clustering analysis by means of the so-called latent allocation variables \citep[see][]{FMW10b,FMW19}. These random variables contain complete information about the partition structure, and despite the ordering of the weights, their finite dimensional distributions can be expressed in terms of expectations of power products of the length variables. Namely, for $\{\x_i\mid \bmu \iid \bmu;\, i \geq 1\}$ where $\bmu = \sum_{j \geq 1}\w_j \delta_{\bxi_j}$ is a proper species sampling process, we define the latent allocation variables $(\d_i)_{i \geq 1}$ through  $\d_i = j$ if and only if $\x_i = \bxi_j$, so that
\begin{equation}\label{eq:d|W}
\l\{\d_i \mid \W \iid \sum_{j \geq 1}\w_j\delta_j; \, i \geq 1\r\}.
\end{equation}
The diffuseness of the base measure implies $\bPi(\x_{1:n})$ is equal almost surely to $\bPi(\d_{1:n})$. Furthermore, the conditional distribution of $(\x_1,\ldots,\x_n)$ given $\bPi(\x_{1:n})$ is that of $(\x_1,\ldots,\x_n)$ given $(\d_1,\ldots,\d_n)$. From \eqref{eq:d|W} we can easily compute for every $n \geq 1$ and any positive integers $d_1,\ldots,d_n$,
\begin{equation}\label{eq:d_pepf}
\Prob[\d_1 = d_1,\ldots,\d_n = d_n] = \Esp\l[\prod_{j=1}^{k}\w_j^{r_j}\r] = \Esp\l[\prod_{j=1}^{k}\v_j^{r_j}(1-\v_j)^{t_j}\r],
\end{equation}
where $k = \max\{d_1,\ldots,d_n\}$, $r_j = \sum_{i=1}^n \Ind_{\{d_i = j\}}$ and $t_j = \sum_{i=1}^n \Ind_{\{d_i > j\}} = \sum_{i>j} r_i$. In particular if the length variables are exchangeable, de Finetti's theorem yields
\[
\Esp\l[\prod_{j=1}^{k}\v_j^{r_j}(1-\v_j)^{t_j}\r] = \int\l\{ \prod_{j=1}^k\int v^{r_j}(1-v)^{t_j}\nu(dv) \r\}\msf{Q}(d\nu)
\]
where $\msf{Q}$ is the law of the directing random measure of $(\v_i)_{i\geq 1}$.

\section{Exchangeable length variables driven by species sampling process}\label{sec:SBESSP}

{In this section we will analyze ESBs with length variables driven by another species sampling process, $\bnu$.} To be precise, we will study general ESBs denoted by $\bmu$, with length variables $\{\v_i\mid \bnu \iid \bnu;\, i \geq 1\}$, for some species sampling process, $\bnu$, over $\l([0,1],\B_{[0,1]}\r)$. We will denote by $\pi_{\nu}$ the EPPF corresponding to $\bnu$, by $\rho_{\nu}$ its tie probability, and by $\nu_0$ the base measure of $\bnu$. Noting that the base measure satisfies $\nu_0 = \Esp[\bnu]$, from the first part of Theorem \ref{theo:exch_SB}, we know that if $(0,\varepsilon)$ belongs to the support of $\nu_0$, for some $0 <\varepsilon < 1$, then $\bmu$ has full support. Furthermore, as $\bnu$ is a species sampling process, by definition its base measure, $\nu_0$, is required to be diffuse, so that it can not have an atom in $0$, thus $\bmu$ must be proper. This gives the following corollary of Theorem \ref{theo:exch_SB}.

{
\begin{cor}\label{cor:sum1}
Let $\bmu$ be an ESB with length variables, $(\v_i)_{i \geq 1}$, that are driven by a species sampling process, $\bnu$, with base measure $\nu_0$. Then $\bmu$ is proper, and if $(0,\varepsilon)$ is contained in the support $\nu_0$, for some $\varepsilon > 0$, $\bmu$ also has full support.
\end{cor}

As mentioned at the beginning of Section \ref{sec:SBE}, Dirichlet and Geometric processes are examples of ESBs, as they have exchangeable length variables, $\{\v_i\mid \bnu \iid \bnu;\, i \geq 1\}$, with  $\bnu = \nu_0 = \Be(1,\theta)$ and $\bnu = \delta_{\v}$, respectively. Clearly, Dirichlet and Geometric processes also belong to the sub-class of ESBs driven by species sampling processes. An appealing advantage of restricting the study to this class, becomes evident when we specialize Theorem \ref{theo:exch_SB_limit_0}. In effect, whilst Theorem \ref{theo:exch_SB_limit_0} modulates the convergence of ESBs to Geometric and Dirichlet processes by means of random probability measures, the following corollary controls the convergence in terms of the underlying tie probability, $\rho_{\nu} = \Prob[\v_1, = \v_2] \in [0,1]$.

\begin{cor}\label{cor:exch_SB_limit}
Consider some diffuse probability measures $\mu_0,\mu^{(1)}_0,\mu^{(2)}_0,\ldots$, over $(S,\B_S)$, and some diffuse probability measures, $\nu_0,\nu^{(1)}_0,\nu^{(2)}_0,\ldots$,  over $\l([0,1],\B_{[0,1]}\r)$. Say that as $n \to \infty$, $\mu^{(n)}_0$ converges weakly to $\mu_0$ and $\nu^{(n)}_0$ converges weakly to $\nu_0$. For $n \geq 1$, let $\rho_{\nu}^{(n)} \in (0,1)$, and consider the ESB $\bmu^{(n)}$ with base measure $\mu^{(n)}_0$, and length variables $\l\{\v^{(n)}_i\mi \bnu^{(n)} \iid \bnu^{(n)};\, i \geq 1\r\}$, where $\bnu^{(n)}$ is a species sampling process with base measure $\nu^{(n)}_0$ and tie probability $\rho_{\nu}^{(n)}$. Also consider the stick-breaking weights of $\bmu^{(n)}$, $\W^{(n)} = \SB\l[\V^{(n)}\r]$, where $\V^{(n)}=\l(\v^{(n)}_i\r)_{i \geq 1}$.
\begin{itemize}
\item[i)] If $\rho_{\nu}^{(n)} \to 0$, as $n \to \infty$, then $\bmu^{(n)}$ converges weakly in distribution to a stick-breaking process, $\bmu$, with base measure, $\mu_0$, and independent length variables $\{\v_i \iid \nu_0;\, i \geq 1\}$. Particularly, if $\nu_0 = \Be(1,\theta)$, $\bmu$ is a Dirichlet process with total mass parameter $\theta$, and $\W^{(n)}$ converges in distribution to the size-biased permuted weights of $\bmu$.
\item[ii)] If $\rho_{\nu}^{(n)} \to 1$, as $n \to \infty$, then  $\bmu^{(n)}$ converges weakly in distribution to a Geometric process, $\bmu$, with base measure $\mu_0$ and length variable $\v \sim \nu_0$, and $\W^{(n)}$ converges in distribution to the decreasingly ordered weights of $\bmu$.
\end{itemize}
\end{cor}
}

{ 
Corollary \ref{cor:exch_SB_limit} states that by means of ESBs featuring species sampling driven length variables, we can approximate with arbitrary precision the distribution of a Geometric process, by making $\rho_{\nu} \to 1$. Alternatively, whenever the length variables have $\Be(1,\theta)$ marginals, by letting $\rho_{\nu} \to 0$, we can arbitrarily approximate Dirichlet priors. Since $[0,1]$ is a set with no gaps, our result sets up a continuous bridge between Geometric and Dirichlet processes, and provides a new interpretation of these random probability measures as opposite extreme points of the class of ESBs. Corollary \ref{cor:exch_SB_limit} also has important consequences when it comes to understanding the ordering of the weights in question. Indeed, as $\rho_{\nu} \to 1$, the weights become more likely to be decreasingly ordered, and in the special case where the length variables have $\Be(1,\theta)$ marginals, the ordering of the weights ranges from a size-biased order to a decreasing order. To understand better how the stochastic ordering of the weights is affected by $\rho_{\nu}$ we provide the following result. 
}

\begin{theo}\label{theo:weights_ord}
{
Consider some exchangeable length variables $\{\v_i\mid \bnu \iid \bnu;\, i \geq 1\}$, where $\bnu$ is a species sampling process with base measure $\nu_0$, EPPF $\pi_{\nu}$, and tie probability $\rho_{\nu}$. Set $\l(\w_j\r)_{j \geq 1} = \SB\l[\l(\v_i\r)_{i \geq 1}\r]$. Then, for every $j \geq 1$,}
\begin{itemize}
{
\item[a)] $\Prob\l[\w_j \geq \w_{j+1}\r] = \rho_{\nu} + (1-\rho_{\nu})\Esp\l[\overrightarrow{\nu_0}(c(\v))\r]$,
where $c(v) = 1 \wedge v(1-v)^{-1}$ for every $v \in [0,1]$,  $\v \sim \nu_0$, and $\overrightarrow{\nu_0}$ is the distribution function of $\nu_0$, that is $\overrightarrow{\nu_0}(x) = \nu_0([0,x])$.
\item[b)] Letting $c$ and $\overrightarrow{\nu_0}$ be as in (a),
\begin{align*}
\Prob[\w_{j}\geq \w_{j+1} \mid \w_1,\ldots,\w_{j}] & = \Prob[\v_{j+1}\leq c(\v_j) \mid \v_1,\ldots,\v_{j}]\\
& = \sum_{i=1}^{\K_j}\frac{\pi_{\nu}\l(\n^{(i)}\r)}{\pi_{\nu}(\n)}\Ind_{\{\v^*_i \leq c(\v_j)\}} + \frac{\pi_{\nu}\l(\n^{(\K_j+1)}\r)}{\pi_{\nu}(\n)}\overrightarrow{\nu_0}(c(\v_j)),
\end{align*}
where $\v^*_1,\ldots,\v^{*}_{\K_j}$ are the distinct values that $\{\v_1,\ldots,\v_j\}$ exhibits, $\n = (\n_1,\ldots\n_{\K_j})$ is given by $\n_{i} = |\{l\leq j:\v_l = \v^{*}_i\}|$,  $\n^{(i)} = (\n_1,\ldots \n_{i-1},\n_i+1,\n_{i+1}, \ldots,\n_{\K_j})$ and  $\n^{(\K_j+1)} = (\n_1,\ldots,\n_{\K_j},1)$.
}
\end{itemize}
\end{theo}

{
A curious highlight about Theorem \ref{theo:weights_ord} is that $\Prob[\w_{j}\geq \w_{j+1}]$ does not depends on $j$, in effect, for every $j \geq 1$, $\Prob[\w_{j}\geq \w_{j+1}]$ is simply a linear combination, modulated by $\rho_{\nu}$, between a quantity determined by $\nu_0$, and one. This means that for a fixed base measure, $\nu_0$, by tuning the tie probability, $\rho_{\nu}$, we can control how close is $\Prob[\w_{j}\geq \w_{j+1}]$ to $\Esp\l[\overrightarrow{\nu_0}(c(\v))\r]$ or one. Theorem  \ref{theo:weights_ord} (b) computes the conditional probability $\Prob[\w_{j}\geq \w_{j+1} \mid \w_1,\ldots,\w_{j}]$, which is completely determined by the EPPF, $\pi_{\nu}$, the base measure $\nu_0$, and the first $j$ length variables. To spell out the corresponding equation, note that given the first $j$ weights, the probability that $\w_{j} \geq \w_{j+1}$ is precisely the probability that $\v_{j+1} = \v^{*}_i$ for some $\v^{*}_i \leq c(\v_j)$ or that $\v_{j+1}$ takes a new value, say $\v^{*}_{\K_j+1}$, which is smaller than $c(\v_j)$.}

The last property we analyze for this type of species sampling processes are finite dimensional distributions of the latent allocation random variables $\d_1,\d_2,\ldots$, as in \eqref{eq:d_pepf}. It is clear that these probabilities can be written in terms of the expectation of power products of the length variables. Equation \eqref{eq:means_samples} explains how to compute these expectations when the length variables are driven by a species sampling process. Hence, the following result is a straightforward consequence of \eqref{eq:means_samples} and \eqref{eq:d_pepf}.

\begin{theo}\label{theo:Ev_d}
Consider the weights sequence, $\W = (\w_j)_{j \geq 1}$, as in \eqref{eq:sb}, with length variables $\{\v_i\mid \bnu \iid \bnu;\, i \geq 1\}$, for some species sampling process, $\bnu$, with base measure $\nu_0$ and EPPF $\pi_{\nu}$. Then, for $\l\{\d_i \mid \W \iid \sum_{j \geq 1}\w_j\delta_j; \, i \geq 1\r\}$, and every $d_1,\ldots,d_n \in \N$,
\begin{equation}\label{eq:Ev_d}
\begin{aligned}
\Prob[\d_1 = d_1,\ldots,\d_n = d_n] = &\sum_{\{A_1,\ldots,A_m\}}\pi_{\nu}(|A_1|,\ldots,|A_m|) \times\\
&  \quad \times\prod_{j=1}^m\l\{\int_{[0,1]} (v)^{\sum_{i \in A_j}r_i}(1-v)^{\sum_{i \in A_j}t_i}\,\nu_0(dv)\r\},
\end{aligned}
\end{equation}
where $r_i = \sum_{l=1}^{n}\Ind_{\{d_l = i\}}$, $t_i = \sum_{l=1}^{n}\Ind_{\{d_l > i\}}$, and where the sum ranges over the set of all partitions of $\{1,\ldots,k\}$, for $k = \max\{d_1,\ldots,d_n\}$.
\end{theo} 

\begin{rem}\label{rem:beta_int}
If $\nu_0$ denotes a $\Be(a,b)$ distribution, the integrals in Theorem \ref{theo:Ev_d} are given by
\begin{align*}
\int_{[0,1]} (v)^{\sum_{i \in A_j}a_i}(1-v)^{\sum_{i \in A_j}b_i}\,\nu_0(dv) & = \frac{\Gamma(a+b)\Gamma(a+\sum_{i\in A_j}a_i)\Gamma(b+\sum_{i\in A_j}b_i)}{\Gamma(a)\Gamma(b)\Gamma(a+b+\sum_{i \in A_j}(a_i+b_i))}.
\end{align*}
for any non-negative integers $(a_i,b_i)_{i \geq 1}$.
\end{rem}

\section{Examples}\label{sec:SBEex}
\subsection{Dirichlet driven stick-breaking processes}\label{sec:SBEDP}

\subsubsection{Theoretical results}

To illustrate the results of Section \ref{sec:SBESSP},  we first concentrate in exchangeable length variables, $\V = (\v_i)_{i \geq 1}$, driven by a Dirichlet process $\bnu$ with total mass parameters $\beta$ and base measure $\nu_0$. {In this case we have that the EPPF, $\pi_\nu$, corresponding to $\bnu$ is given by
\begin{equation}\label{eq:eppf_dp}
\Prob[\bPi(\v_{1:k}) = A]= \pi_\nu(|A_1|,\ldots,|A_m|)= \frac{\beta^{m}}{(\beta)_k}\prod_{i=1}^m(|A_i|-1)!
\end{equation}
for every partition $A = \{A_1,\ldots,A_m\}$ of $\{1,\ldots,k\}$ \citep{E72, A74}. In particular we can compute the underlying tie probability $\rho_{\nu} = \pi(2) = 1/(1+\beta)$.} Motivated by Corollaries \ref{cor:exch_SB_beta} and  \ref{cor:exch_SB_limit}, we will further study the case $\nu_0 = \Be(1,\theta)$. Indeed, this assumption guaranties that the corresponding ESB, $\bmu$, is proper and has full support, and as a consequence of Corollary \ref{cor:exch_SB_limit}, it allows us to recover Geometric and Dirichlet processes in the weak limits as $\beta \to 0$ and $\beta \to \infty$, respectively. Any such ESB, $\bmu$, will be called a Dirichlet driven stick-breaking process (DSB) with parameters $(\beta,\theta,\mu_0)$, and its weights, $\W = \SB[\V]$, will be referred to as  Dirichlet driven stick-breaking weights (DSBw) with parameters $(\beta,\theta)$. Next, we specialize the results of Section \ref{sec:SBESSP} to DSBs.

\begin{cor}\label{cor:limit_DSB}
Let $\mu_0,\mu^{(1)}_0,\mu^{(2)}_0$ be diffuse probability measures over $(S,\B_S)$ such that $\mu^{(n)}_0$ converges weakly to $\mu_0$, as $n \to \infty$. For each $n \geq 1$ consider $\beta^{(n)},\theta^{(n)} \in (0,\infty)$, with $\theta^{(n)} \to \theta$ in $(0,\infty)$. Let $\bmu^{(n)}$ be a \emph{DSB} with parameters $\l(\beta^{(n)},\theta^{(n)},\mu^{(n)}_0\r)$, and let $\W^{(n)}$ the corresponding \emph{DSBw} with parameters $\l(\beta^{(n)},\theta^{(n)}\r)$. 
\begin{itemize}
\item[i)] If $\beta^{(n)} \to \infty$, as $n \to \infty$, then $\bmu^{(n)}$ converges weakly in distribution to a Dirichlet process, $\bmu^{(\infty)}$, with total mass parameter $\theta$ and base measure $\mu_0$, and $\W^{(n)}$ converges in distribution to the size-biased permutation of the weights of $\bmu^{(\infty)}$.
\item[ii)]  If $\beta^{(n)} \to 0$, as $n \to \infty$, then $\bmu^{(n)}$ converges weakly in distribution to $\bmu^{(0)}$, where $\bmu^{(0)}$ stands for a Geometric process with base measure $\mu_0$ and length variable $\v^{*} \sim \Be(1,\theta)$, and $\W^{(n)}$ converges in distribution to the decreasingly ordered weights of $\bmu^{(0)}$.
\end{itemize}
\end{cor}

Corollary \ref{cor:limit_DSB} follows immediately from Corollary \ref{cor:exch_SB_limit} by substituting $\rho^{(n)}_{\nu} = 1/\l(1+\beta^{(n)}\r)$. As to the ordering of DSBw's, we have the following corollary of Theorem \ref{theo:weights_ord}.

\begin{cor}\label{cor:weights_ord}
{
Fix $\theta > 0$, and for each $\beta > 0$, consider a \emph{DSBw}, $\l(\w_j\r)_{j \geq 1}$, with parameters $(\beta,\theta)$. Let us denote by $\2F1$ to the Gauss hypergeometric function. Then, for every $j \geq 1$,}
\begin{itemize}
{
\item[a)]
$\displaystyle 
\Prob\l[\w_j \geq \w_{j+1}\r] = 1 - \frac{\2F1(1,1;\theta+2,1/2)\beta\theta}{2(\beta+1)(\theta+1)},\,
$
for every $\beta > 0$.}
{
\item[b)] 
$\displaystyle
\Prob[\w_{j}\geq \w_{j+1} \mid \w_1,\ldots,\w_{j}] = \frac{1}{\beta+j}\l\{\sum_{i \in \mbf{A}_j}\n_i + \beta\l[1-\l(1-c(\v_j)\r)^{\theta}\r]\r\},
$
where $c(v) = 1\wedge v(1-v)^{-1}$, $\v^*_1,\ldots,\v^{*}_{\K_j}$ are the distinct values that $\{\v_1,\ldots,\v_j\}$ exhibits, $\mbf{A}_j = \{i < \K_j: \v^{*}_i \leq c(\v_j)\}$ and $\n_{i} = |\{l\leq j:\v_l = \v^{*}_i\}|$, for every $i \leq \K_j$.}
\end{itemize}
\end{cor}

\begin{rem}
In the context of Corollary \ref{cor:weights_ord}, if $\theta = 1$, we even obtain
\[
\Prob\l[\w_j \geq \w_{j+1}\r] = \frac{1+\beta\log(2)}{1+\beta}.
\]
\end{rem}

{
Other quantities that simplify nicely for DSBs are the finite dimensional distributions of latent allocation variables. For instance, from \eqref{eq:eppf_dp}, using Remark \ref{rem:beta_int} and recalling that $\Gamma(x+n)/\Gamma(x) = \prod_{i=0}^{n-1}(x+i) = (x)_n$, Theorem \ref{theo:Ev_d} specializes the as follows.
}

\begin{cor}\label{cor:Ev_DSB}
Let $\W = (\w_j)_{j \geq 1}$ be a DSBw with parameters $(\beta,\theta)$, and consider the allocation variables $\l\{\d_i \mid \W \iid \sum_{j \geq 1}\w_j\delta_j; \, i \geq 1\r\}$. Then for any positive integers $d_1,\ldots,d_n$,
\begin{equation}\label{eq:Ev_d_DBS}
\begin{aligned}
\Prob[\d_1 = d_1,\ldots,\d_n = d_n] = \sum_{\{A_1,\ldots,A_m\}}\frac{\l(\beta\theta\r)^{m}}{(\beta)_k}\prod_{j=1}^m\frac{(|A_j|-1)!\,(\sum_{i\in A_j}r_i)!}{(\theta+\sum_{i \in A_j}t_{i})_{1+\sum_{i \in A_j}r_i}},
\end{aligned}
\end{equation}
where $r_i = \sum_{l=1}^{n}\Ind_{\{d_l = i\}}$, $t_i = \sum_{l=1}^{n}\Ind_{\{d_l > i\}}$ and sum ranges over the set of all partitions of $\{1,\ldots,k\}$, for $k = \max\{d_1,\ldots,d_n\}$.
\end{cor}

\subsubsection{Distribution of the number of groups}

A random variable that encloses important information about a species sampling process $\bmu$ is the number of distinct values $\K_n = |\bPi(\x_{1:n})|$, that a sample, $\{\x_i\mid \bmu \iid \bmu;\, i \geq 1\}$ exhibits. For certain species sampling processes, the distribution of $\K_n$ is analytically available  \citep{LMP07b,DeBlasi2015}. However, characterizing this distribution is generally not an easy task. Nonetheless, whenever one can sample from the finite dimensional distributions of the length variables, $(\v_1,\ldots,\v_m)$, and therefore the weights, $(\w_1,\ldots,\w_m)$, drawing samples from $\K_n$ is relatively simple. First sample $\{\u_i \iid \Un(0,1);\, i \geq 1\}$, and  $\w_1,\ldots\w_m$ up the first index, $m$, that satisfies $\sum_{j=1}^{m} \w_j > \max_{k}\u_k$. Define $\d_k = i$, if and only if $\sum_{j=1}^{i-1}\w_j \leq \u_k < \sum_{j=1}^{i}\w_j$, with the convention that the empty sum equals zero. Finally, note that the number of distinct values in $\{\d_1,\ldots,\d_n\}$, is precisely a sample from $\K_n$. 

Denote by $\K_n^{(\beta,\theta)}$ the random variable, $\K_n$, corresponding to a DSB with parameters $(\beta,\theta,\mu_0)$, for simplicity when $\beta = 0$ and $\beta = \infty$ we refer to a Geometric and Dirichlet process, respectively. In Figure \ref{FigKn}, we exhibit the distribution of $\K_n^{(\beta,\theta)}$ for different choices of $\beta$ and $\theta$ and for $n = 20$. Here we illustrate the asymptotic results stated in Corollary \ref{cor:exch_SB_limit} and Corollary \ref{cor:limit_DSB}, by means of $\K_n^{(\beta,\theta)}$. In fact we have a graphical representation of how as $\beta \to 0$ and $\beta \to \infty$, $\K_n^{(\beta,\theta)} \to \K_n^{(0,\theta)}$ and $\K_n^{(\beta,\theta)} \to \K_n^{(\infty,\theta)}$, in distribution. {We also observe that an increment on $\beta$ (decrement of $\rho_{\nu}$) contributes to the distribution of $\K_n^{(\beta,\theta)}$ with a smaller mean and variance. Conversely, decreasing the value of $\beta$, impacts the prior distribution of $\K_n^{(\beta,\theta)}$ with a larger mean and variance, and a heavier right tail, say less informative. We also see that, just as it occurs for the Dirichlet process, for DSBs with the value of $\beta$ fixed, increasing the parameter $\theta$ makes the distribution of $\K_n^{(\beta,\theta)}$ favour larger values. This behaviour can be explained by recalling that marginally $\v_i \sim \Be(1,\theta)$, meaning that larger values of $\theta$ lead to smaller values of the length variables which translates to the smaller values of the first weights, this then leads to more diversity in samples of exchangeable sequences driven by the corresponding DSB.

\begin{figure}[H]
\centering
\includegraphics[scale=0.48]{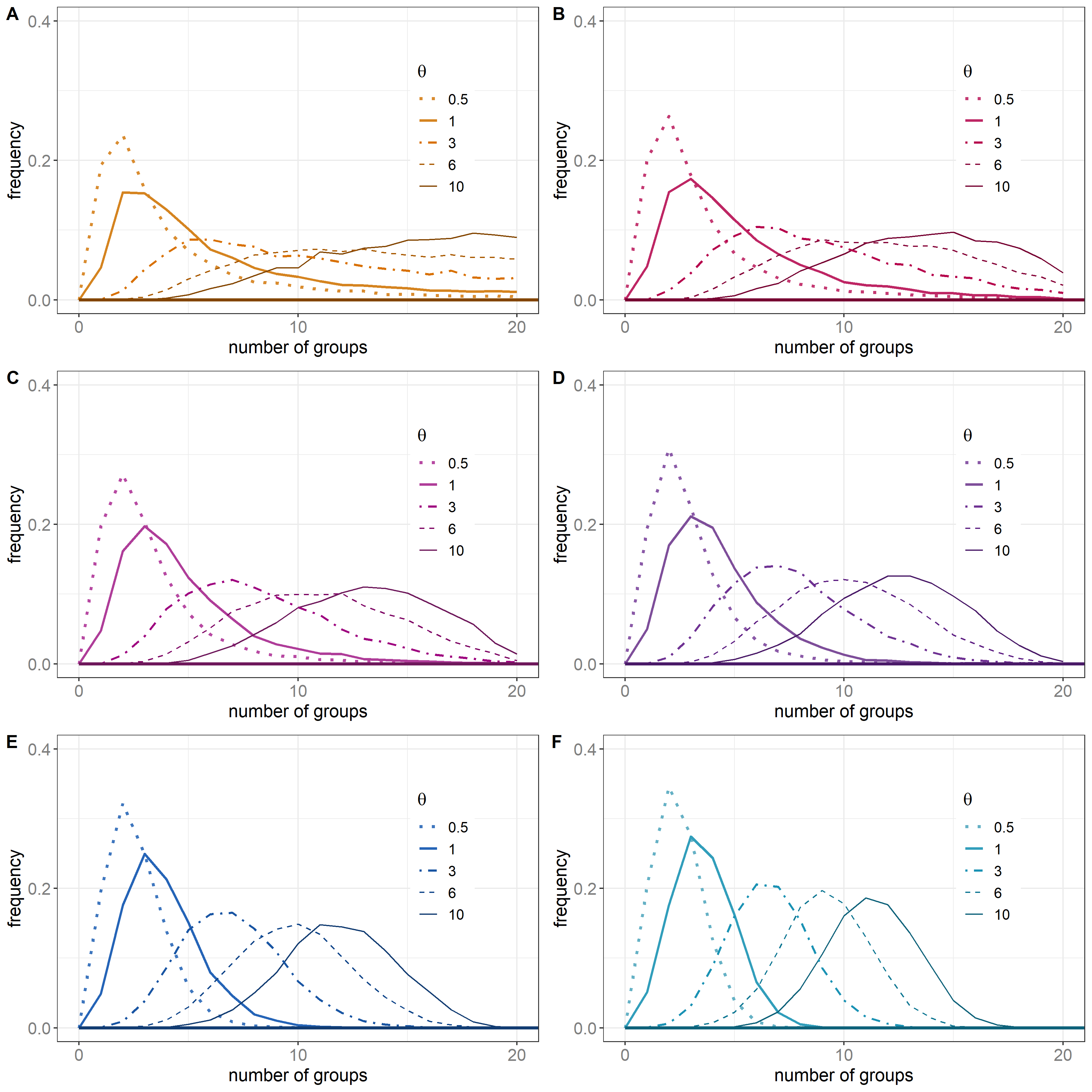} 
\caption{Frequency polygons of samples of size $10000$ from the distribution of $\K_{20}$, corresponding to DSBs with parameters $\beta$ in $\{0,0.5,1,10,100,\infty\}$ ($\msf{A}-\msf{F}$, respectively), and varying $\theta$ in $\{0.5,1,3,6,10\}$.\label{FigKn}}
\end{figure}

Figure \ref{FigEKasymp} provides numerical approximations of the mapping $n \mapsto \Esp\l[\K_n^{(\beta,\theta)}\r]$ for $n \in \{1,\ldots,200\}$, and for distinct values of $\beta$ and $\theta$. We can appreciate that for a fixed value of $\beta = (1-\rho_{\nu})/\rho_{\nu}$, larger values of $\theta$ lead to a faster growth of $\Esp\l[\K_n^{(\beta,\theta)}\r]$ as $n$ increases. This is consistent with the analysis in Figure \ref{FigKn}, in the sense that larger values of $\theta$ translate to $\K_n^{(\beta,\theta)}$ being more prone to take bigger values. Alternatively, if we fix the parameter $\theta$, we see that an increment on $\rho_{\nu} = 1/(\beta+1)$, leads to a faster growth of $\Esp\l[\K_n^{(\beta,\theta)}\r]$ with respect to $n$. Indeed, for the Dirichlet model ($\rho_{\nu} = 0$) we see that $\Esp\l[\K_n^{(\infty,\theta)}\r]$ is the one with the slowest growth, and at the other extreme, for the Geometric process ($\rho_{\nu} = 1$), $\Esp\l[\K_n^{(0,\theta)}\r]$ exhibits the fastest growth. This analysis suggests that for $0 < \rho_{\nu} = 1/(\beta+1) < 1$, the rate at which $\Esp\l[\K_n^{(\beta,\theta)}\r]$ increases is bounded by those corresponding to a Geometric and a Dirichlet model with the same value of $\theta$. Namely, it seems sensible to propose the conjecture $\Esp\l[\K_n^{(\infty,\theta)}\r]\leq\Esp\l[\K_n^{(\beta_2,\theta)}\r]\leq\Esp\l[\K_n^{(\beta_1,\theta)}\r]\leq\Esp\l[\K_n^{(0,\theta)}\r]$ for real numbers $0 < \beta_1 < \beta_2$.}

\begin{figure}[H]
\centering
\includegraphics[scale=0.48]{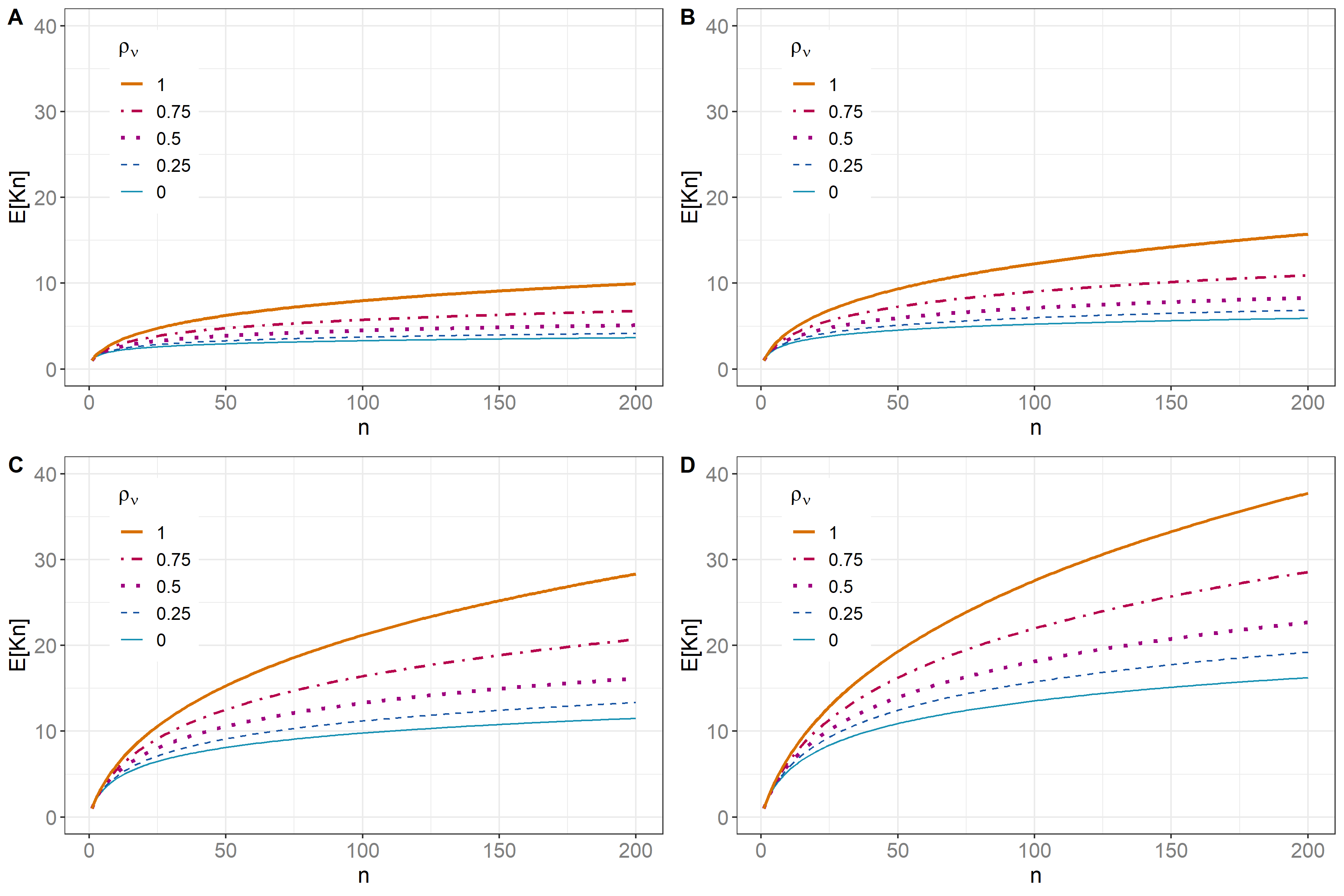} 
\caption{Numerical approximation of the mapping $n \mapsto\Esp[\K_n]$, for $n \in \{1,\ldots,200\}$, corresponding to DSBs with parameters $\theta$ in $\{0.5,1,2.5,4\}$ ($\msf{A}-\msf{D}$, respectively), and for each fixed value of $\theta$, we vary $\beta$ in $\{0,1/3,1,3,\infty\}$ ($\rho_{\nu} \in \{1,0.75,0.5,0.25,0\}$, respectively).\label{FigEKasymp}}
\end{figure}

\subsection{Models beyond Dirichlet driven stick-breaking processes}

The generality of Section \ref{sec:SBESSP} allows us to construct a great variety of  new Bayesian non-parametric priors. In fact, for every known species sampling process, $\bnu$, with available tie probability, $\rho_{\nu}$, the analysis in Section \ref{sec:SBESSP} can be carried out, thus leading to a new model. In addition, if the expression for the EPPF, $\pi_{\nu}$, is manageable, relevant clustering probabilities can be recovered by inserting $\pi_{\nu}$ into the finite sum in \eqref{eq:Ev_d}. Such is the case of Gibbs-type priors \citep[e.g.][]{DeBlasi2015}, and normalized random measures with independent increments \citep{Regazzini03,JLP09}.

To provide another concrete example, let us consider the case where $\bnu$ is the two-parameter Pitman-Yor process \citep{PY92,PPY92} with parameters $\alpha \in [0,1)$ and $\beta > -\alpha$. For this species sampling process the EPPF is 
\begin{equation}\label{eq:eppf_py}
\pi_{\nu}(n_1,\ldots,n_k) = \frac{(\beta+\alpha)_{k-1 \uparrow \alpha}\prod_{j=1}^{k}(1-\alpha)_{n_j-1}}{(\beta +1)_{n-1}},
\end{equation}
where $n = \sum_{j=1}^k n_j$, $(x)_{m \uparrow a} = \prod_{i=0}^{m-1}(x+ia)$ and $(x)_m = (x)_{m \uparrow 1}$. This implies that the prediction rule for $\{\v_i\mid \bnu \iid \bnu;\, i \geq 1\}$, becomes
\begin{equation}\label{eq:pred_py}
\Prob[\v_{n+1}\in \cdot \mid \v_1,\ldots\v_n] = \sum_{j=1}^{\k} \frac{\n_j-\alpha}{n+\beta}\delta_{\v^*_j}+\frac{\beta+\k\alpha}{n+\beta}\nu_0,
\end{equation}
where $\v^*_1,\ldots\v^*_{\k}$ are the distinct values in $\{\v_1,\ldots,\v_n\}$, and $\n_{j} = |\{i\leq n:\v_i = \v^{*}_j\}|$, for every $n \geq 1$. As to the tie probability of $\bnu$ we get
\begin{equation}\label{eq:rho_py}
\rho_{\nu} = \pi_{\nu}(2) = \frac{1-\alpha}{\beta + 1}.
\end{equation}
Inserting \eqref{eq:eppf_py} into \eqref{eq:Ev_d}, for the case where the base measure $\nu_0 = \Be(1,\theta)$, we obtain the finite dimensional distributions of the latent allocation variables
\begin{align*}
\Prob[\d_1 = d_1,\ldots,\d_n = d_n] = \sum_{\{A_1,\ldots,A_m\}}\frac{\theta^{m}(\beta+\alpha)_{m-1 \uparrow \alpha}}{(\beta +1)_{k-1}}\prod_{j=1}^m\frac{(1-\alpha)_{|A_j|-1}\,(\sum_{i\in A_j}r_i)!}{(\theta+\sum_{i \in A_j}t_{i})_{1+\sum_{i \in A_j}r_i}}.
\end{align*}
where $r_i = \sum_{l=1}^{n}\Ind_{\{d_l = i\}}$, $t_i = \sum_{l=1}^{n}\Ind_{\{d_l > i\}}$, and where the sum ranges over the set of all partitions of $\{1,\ldots,k\}$, for $k = \max\{d_1,\ldots,d_n\}$.

For the corresponding weights, $(\w_j)_{j \geq 1}$, substituting \eqref{eq:rho_py} into Theorem \ref{theo:weights_ord} yields
\[
\Prob[\w_{j} \geq \w_{j+1}] = \frac{1-\alpha +(\beta+\alpha)\Esp\l[\overrightarrow{\nu_0}(c(\v))\r]}{\beta + 1},
\]
for every $j \geq 1$, and where $\overrightarrow{\nu_0}$ and $c$ are as stated in the theorem. For the special case $\nu_0 = \Be(1,\theta)$, the probability in question simplifies to
\[
\Prob[\w_{j} \geq \w_{j+1}] = 1-\frac{\2F1(1,1;\theta+2,1/2)(\beta+\alpha)\theta}{2(\beta + 1)(\theta+1)},
\]
and if $\theta = 1$ we even get $\Prob[\w_{j} \geq \w_{j+1}] = [1-\alpha +(\beta+\alpha)\log(2)]/[\beta + 1]$. In general, by \eqref{eq:rho_py}, as $\alpha \to 1$ and $\beta \to \beta' \in [1,\infty)$ or $\beta \to \infty$ and $\alpha \to \alpha' \in [0,1)$ we get $\rho_{\nu} \to 0$. Alternatively, as $\alpha \to \alpha' \in [0,1)$ and $\beta \to -\alpha'$, the tie probability $\rho_{\nu} \to 1$. Thus, Corollary \ref{cor:exch_SB_limit} assures that by means of Pitman-Yor driven ESBs we can approximate (weakly in distribution) Dirichlet and Geometric process. This is not true for all choices of $\bnu$, in fact if $\bnu$ is a normalized inverse-Gaussian random measure, with total mass parameter $\beta > 0$, as proved by \cite{LMP05}, its tie probability is $\rho_{\nu} = 2^{-1}[1+\beta^2 e^{\beta}E_1(\beta)-\beta]$, where $E_1(\beta) = \int_{\beta}^{\infty} x^{-1}e^{-x}dx$ is the exponential integral. Using the inequality
\[
\frac{e^{-\beta}}{2}\log\l(1+\frac{2}{\beta}\r) < E_1(\beta) < e^{-\beta}\log\l(1+\frac{1}{\beta}\r),
\]
it can be shown as $\beta \to \infty$, $\rho_{\nu} \leq 0$, and as $\beta \to 0$, $\rho_{\nu} \to c \leq 1/2$. Thus Geometric processes can not be recovered as weak limits of ESBs with normalized inverse-Gaussian processes as the directing random measure of the length variables.

{There are also interesting choices of $\bnu$, outside  Bayesian non-parametric priors. For example, one might consider the species sampling process with finitely many atoms, $\bnu = \bm{\alpha}\sum_{j = 1}^{\kappa}\p_j\delta_{\v^*_j} +\l(1-\bm{\alpha}\r)\nu_0$, for some $\kappa \in \N$ and where $\p_j \geq 0$, $\sum_{j = 1}^{\kappa}\p_j = 1$ and $\bm{\alpha}$ is an independent random variable taking values in $[0,1]$.} Depending on the distribution of $(\p_j)_{j=1}^{\kappa}$ and $\bm{\alpha}$, the EPPF, $\pi_{\nu}$, could be relatively simple to derive, as well as the tie probability which can be computed through $\rho_{\nu} = \Esp\l[\bm{\alpha}^2\r]\sum_{j=1}^{\kappa}\Esp\l[\p_j^2\r]$. 

\section{Illustrations}\label{sec:Illust}

In Bayesian non-parametric statistics it is common to model data, $\Y = (\y_1,\ldots,\y_m)$, that features no repetitions, as if sampled from $\{\y_k\mid \x_k\} \sim G(\cdot\mid \x_k)$ independently for $k \geq 1$. Here $G$ denotes a diffuse probability kernel from the Polish space, $S$, where $\x_k$ takes values, into the  Polish space, $T$, where $\y_k$ takes values. We further assume that $G(\cdot\mid s)$ has a density for every $s \in S$ and that $(\x_1,\x_2,\ldots)$ is exchangeable and driven by a proper species sampling process $\bmu = \sum_{j \geq 1}\w_j\delta_{\bxi_j}$. In terms of the law of $\Y$, this is equivalent to model the sequence as if it was conditionally i.i.d. sampled from random mixture
\begin{equation}\label{eq:mixture}
\bPhi = \int G(\cdot\mid s)\bmu(ds) = \sum_{j \geq 1}\w_j\,G(\cdot\mid\bxi_j).
\end{equation}
Hereinafter we call $\bPhi$ an ESB mixture or a DSB mixture whenever $\bmu$ is an ESB or a DSB. In general, if $G(\cdot \mid s_n)$ converges weakly to $G(\cdot \mid s)$, as $s_n \to s$ in $S$, the mapping
\[
\sum_{j \geq 1}w_j\delta_{s_j} \mapsto \int G(\cdot\mid s)d\mu(ds) = \sum_{j \geq 1}w_j G(\cdot \mid s_j) 
\]
is continuous with respect to the weak topology (Corollary \ref{lem:dL_cont_map_1} in the appendix). This means that analogous convergence results to those in Theorem \ref{theo:exch_SB_limit_0} and Corollaries \ref{cor:exch_SB_limit} and \ref{cor:limit_DSB} hold for ESB and DSB mixtures. In this context, $\K_n = |\bPi(\x_{1:n})|$ represents the number of mixture components that are significant in the sample, that is the number of elements in $\{G(\cdot\mid\bxi_j)\}_{j \geq 1}$ for which there exist $k \in \{1,\ldots,m\}$ such that $\y_k$ was ultimately sampled from $G(\cdot\mid\bxi_j)$. Details of an MCMC algorithm for density estimation by means of ESB mixtures are provided in Section \ref{sec:MaP_EaP} of the appendix.

We designed two experiments, one consists in estimating the density of univariate data by means of the usual expected a posterior (EAP) estimator. Here we will adjust six DSB mixtures, where the parameter $\beta = (1-\rho_{\nu})/\rho_{\nu}$ is fixed to distinct values. The main objective of this test is to analyze the posterior impact of Corollary \ref{cor:limit_DSB}, meaning that we expect to observe that when $\rho_{\nu} = 1/(\beta+1)$ is small, the posterior estimators behave similar to those of a Dirichlet prior, and when $\rho_{\nu}$ is close to one, the results resemble those of a Geometric prior. For the second experiment we work with bivariate data, and focus on estimating its density by means of the EAP and the maximum a posterior (MAP) estimators, we will also estimate the clusters of the data points using the MAP estimator (see Section \ref{sec:MaP_EaP} of the appendix, for details). In this  second study, we will assign a prior distribution to the underlying tie probability $\rho_{\nu}$, allowing the model to determine which values of $\rho_{\nu}$ suit better the dataset, given that the rest of the parameters and hyper-parameters are fixed.

\subsection{Results for DSB mixtures with fixed tie probability}

For this experiment we simulated observations $(\y_k)_{k=1}^{200}$ from a mixture of seven Normal distributions, and estimate the density of the data through six distinct DSB mixtures with parameters $(\beta,\theta,\mu_0)$. For each of the six mixtures we consider a Gaussian kernel with random location and scale parameters, i.e $G(\y|\bxi_j) = \msf{N}(\y|\m_j,\bm{\tau}_j^{-1})$ and $\mu_0(\bxi_j) = \msf{N}(\m_j|\mu,(\lambda\bm{\tau}_j)^{-1})\Ga(\bm{\tau}_j|a,b)$, where $a = b = 0.5$, $\lambda = 1/100$ and $\mu =n^{-1}\sum_{k=1}^n\y_k$. Each of the six DSB mixtures features a distinct of $\beta = (1-\rho_{\nu})/\rho_{\nu}$ and all share $\theta = 1$. 

\begin{figure}[H]
\centering
\includegraphics[scale=0.48]{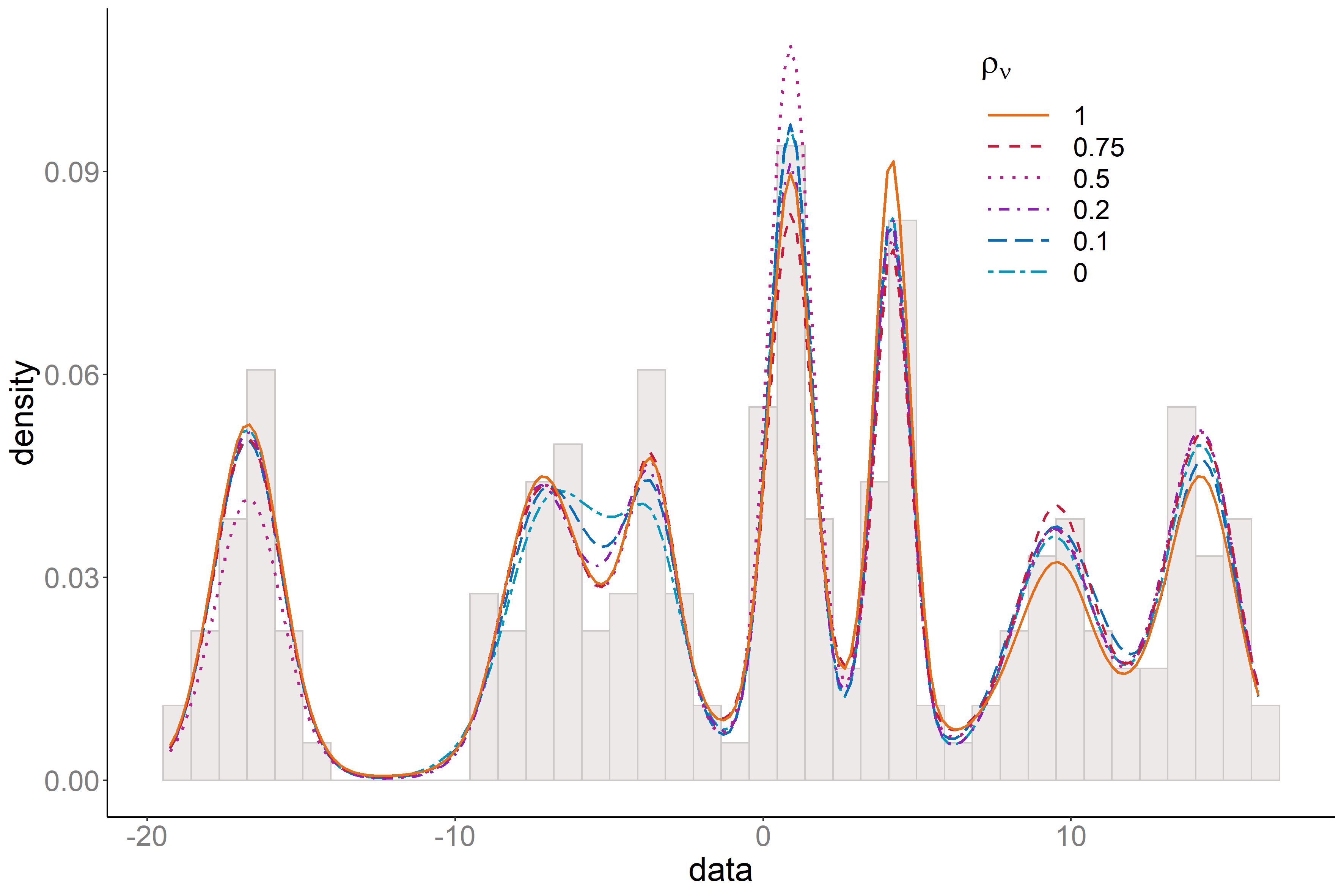} 
\caption{Estimated densities, taking into account $4000$ iterations of the Gibbs sampler, skipping $4$ iterations, after a burn-in period of $7000$, for six DSB mixtures with parameter $\beta \in \{0,1/3,1,4,9,\infty\}$ ($\rho_{\nu} \in \{0.75, 0.5, 0.2, 0.1\}$, respectively).\label{Fig7-D}}
\end{figure}

In Figure \ref{Fig7-D} we can observe that all models do a good job estimating the density. The Dirichlet process ($\rho_{\nu} = 0$) and the DSB with $\rho_{\nu} = 0.1$ struggle more than the other models to differentiate the second and third modes from left to right, this can be due to the initial election of the parameter $\theta$ and the fact that a priori the DSB with $\rho_{\nu} = 0.1$ behaves similarly to a Dirichlet process. In Figure \ref{Fig7-D} we can also observe that it is at the high density areas that the estimated density from model to model varies slightly. 

\begin{figure}[H]
\centering
\includegraphics[scale=0.48]{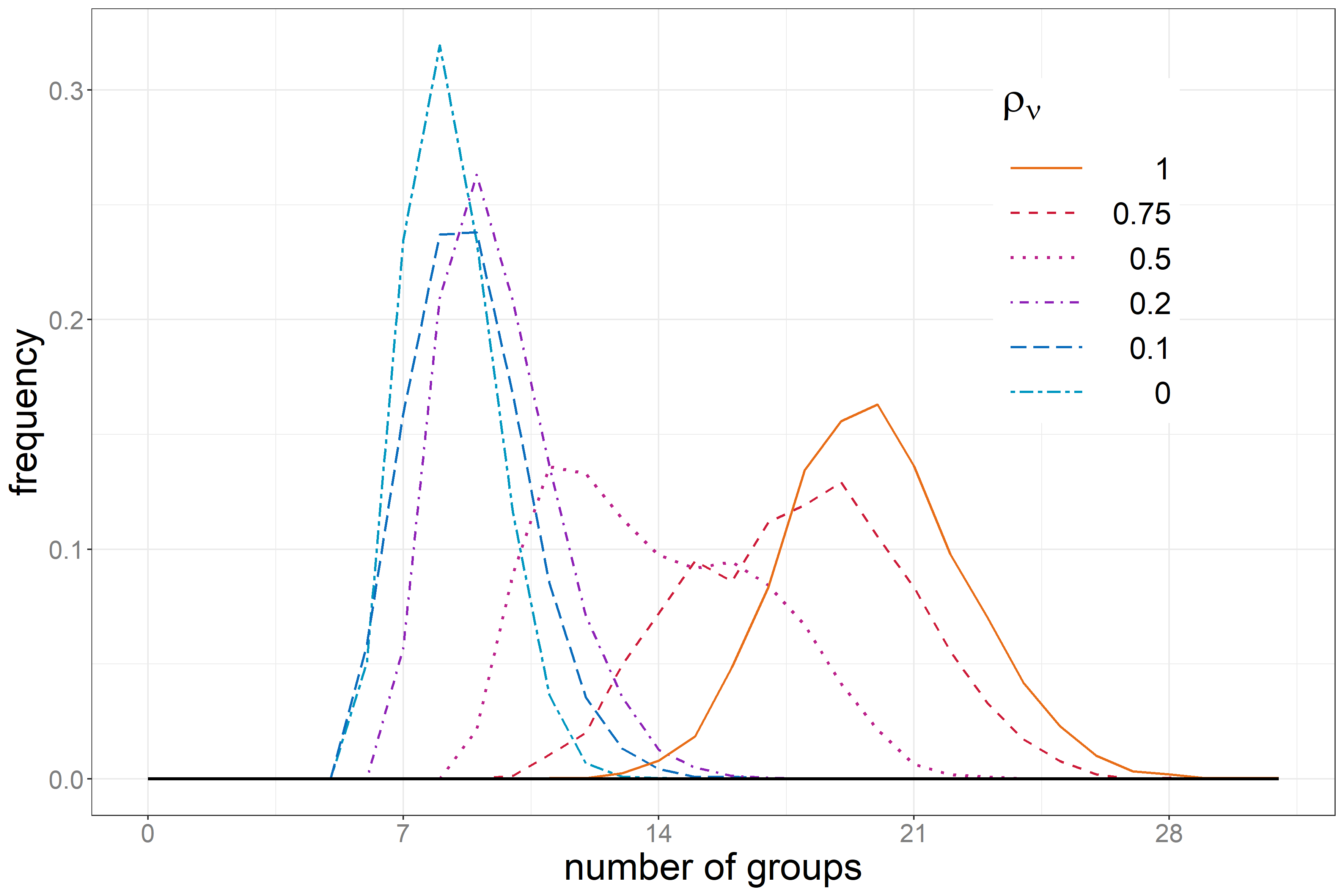} 
\caption{Frequency polygons corresponding to the posterior distribution of $\K_{200}$, for the six DSB mixtures with parameter $\beta = 0,1/3,1,4,9,\infty$ ($\rho_{\nu} = 0.75, 0.5, 0.2, 0.1$, respectively)\label{Fig7-Km}}
\end{figure}

Figure \ref{Fig7-Km} illustrates the posterior distribution of $\K_n$, with $n = 200$, for each of the six DSB mixtures implemented. Here we see that the Dirichlet process and the DSBs with a smaller value of $\rho_{\nu}$, give high probability to numbers close to seven, which is the true number of components of the mixture from which the data was sampled. In contrast, as the parameter $\rho_{\nu}$ approaches one, we observe that the models tend to give higher probability to larger values through the posterior distribution of $\K_n$, this means that these models use more components to provide the estimations illustrated in Figure \ref{Fig7-D}. Indeed, since  Geometric weights decrease at a constant rate, in order to estimate the size and shape of some components, the model is forced to overlap many small components. If we were interested in clustering the data points, this can be a disadvantage of DSB mixtures with a large values of $\rho_{\nu}$, as it is likely that the number of clusters will be overestimated. However, if we are only interested in density estimation this feature actually makes the models that behave similar to Geometric processes more likely to capture subtle changes in the histogram of the data set. Overall we see that the results are consistent with Corollaries \ref{cor:exch_SB_limit} and \ref{cor:limit_DSB} in the sense that as $\rho_{\nu} \to 0$, the results provided by the DSB mixtures are similar to those given by a Dirichlet prior, and when $\rho_{\nu} \to 1$, they are closer to those provided by a Geometric prior.

\subsection{Results for DSB mixtures with random tie probability}\label{sec:ill_rand}

For this experiment, we simulated $510$ data points from a paw-shaped mixture of seven Normal distributions. Here we adjust a Dirichlet mixture with total mass parameter $\theta$, a Geometric mixture with length variable $\v \sim \Be(1,\theta)$ and a DSB with parameters $(\beta,\theta,\mu_0)$, where $\beta = (1-\rho_{\nu})/\rho_{\nu}$ and $\rho_{\nu}\sim \msf{Unif}(0,1)$. For all mixtures we assume a bivariate Gaussian kernel, i.e. $G(\y|\bxi_j) = \msf{N_2}(\y|\m_j,\bm{\Sigma}_j)$, and a Normal-inverse-Wishart prior for $\bxi_j = (\m_j,\bm{\Sigma}_j)$, so that $\mu_0(\bxi_j) = \msf{N_2}(\m_j\mid \mu,\lambda^{-1}\bm{\Sigma}_j)\msf{W}^{-1}(\bm{\Sigma}_j\mid \mathrm{P},\nu)$. In all cases the hyper-parameters were fixed to $\theta = 1$, $\mu = n^{-1}\sum_{k=1}^{n}\y_k$, $\lambda = 1/100$, $\nu = 2$ and $\mathrm{P}$ equal to the identity matrix.

\begin{figure}[H]
\centering
\includegraphics[scale=0.48]{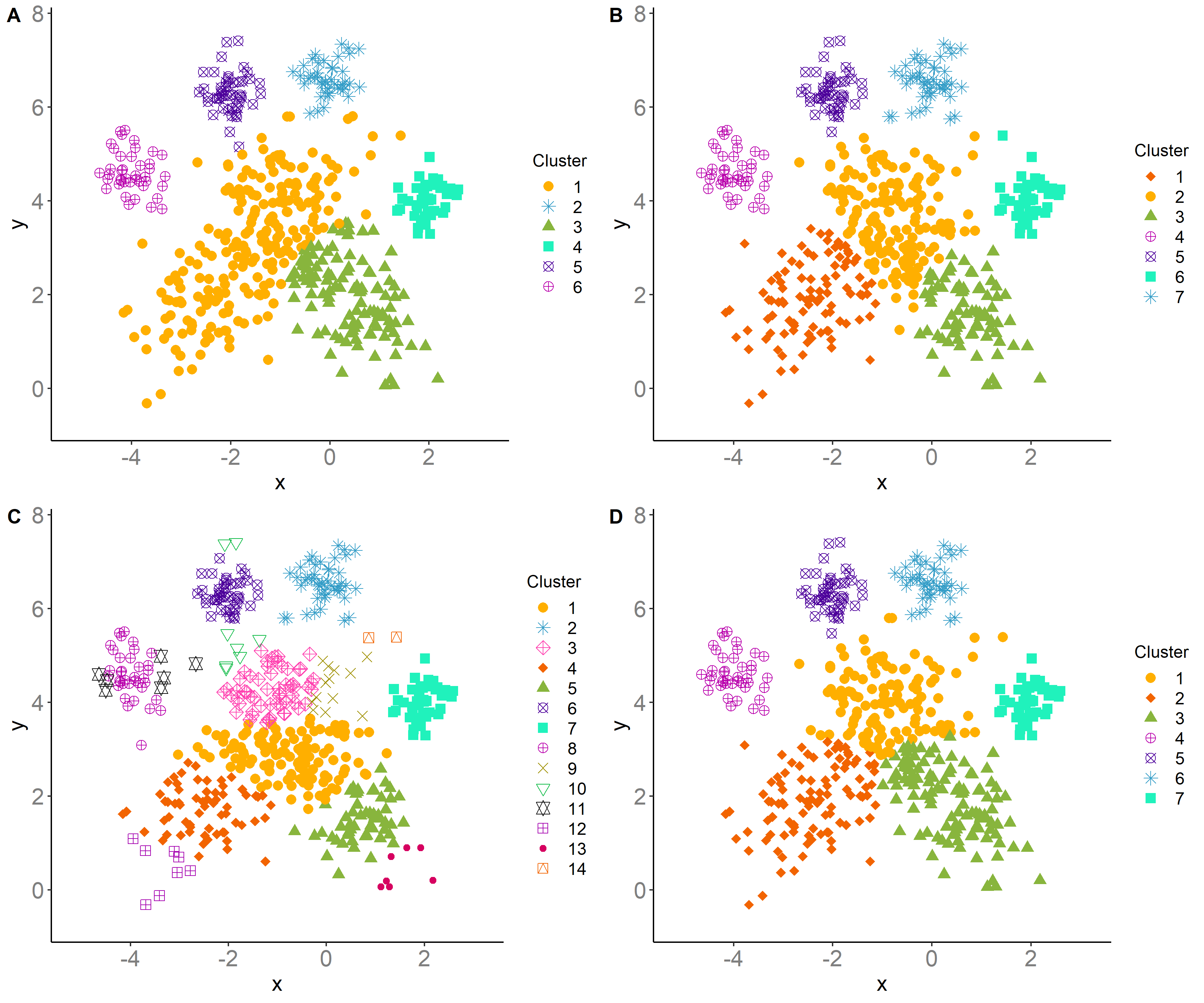} 
\caption{Estimated clusters using the MAP estimator, taking into account $8000$ iterations of the Gibbs sampler after a burn-in period of $2000$ iterations, according to the Dirichlet prior $(\msf{A})$ a DSB prior $(\msf{B})$, a Geometric prior $(\msf{C})$ and the true model $(\msf{D})$. \label{FigPaw3-cl}}
\end{figure}

In Figure \ref{FigPaw3-cl} we see that while the DSB recovers exactly seven clusters as there are according to the true model, the Dirichlet process merges two clusters of the true model, and the Geometric model overestimates the number of clusters. This figures illustrates that the generality inherent DSBs allows to balance features of Dirichlet and Geometric processes. This is also reflected through the MAP estimators of the density, presented in Figure \ref{FigPaw3-2d-MaP} (and Figure \ref{FigPaw3-3d-MaP} in the appendix). 

\begin{figure}[H]
\centering
\includegraphics[scale=0.48]{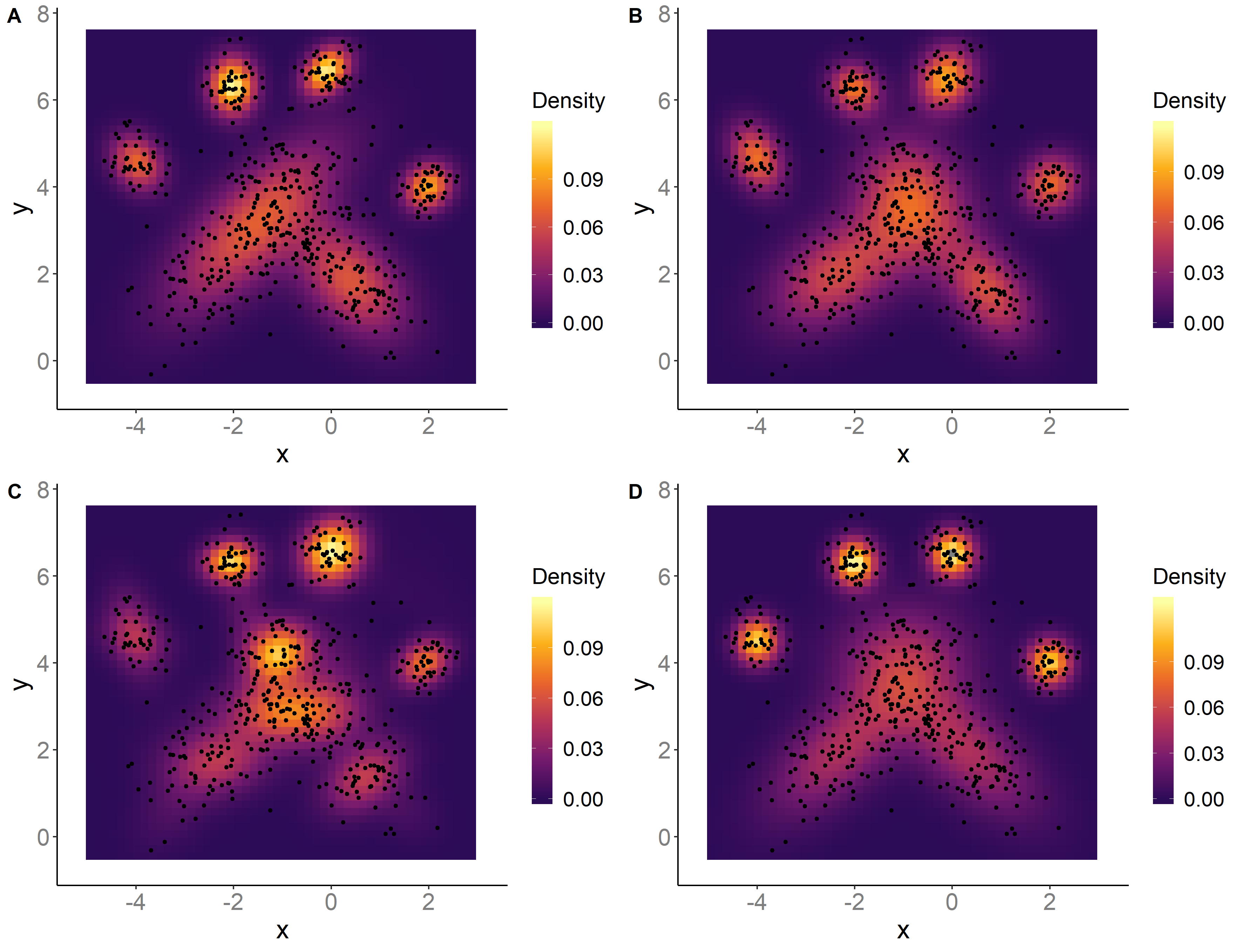} 
\caption{Estimated densities of the data points using the MAP, taking into account $8000$ iterations of the Gibbs sampler after a burn-in period of $2000$ iterations, according to the Dirichlet prior $(\msf{A})$ a DSB prior $(\msf{B})$  and a Geometric prior $(\msf{C})$. $\msf{D}$ shows the true density. \label{FigPaw3-2d-MaP}}
\end{figure}

%\begin{figure}[H]
%\centering
%\includegraphics[scale=0.85]{paw3_3d_MaP} 
%\caption{$3$D view of the estimated densities using the MAP, taking into account $8000$ iterations of the Gibbs sampler after a burn-in period of $2000$ iterations, according to the Dirichlet prior $(\msf{A})$ a DSB prior $(\msf{B})$  and a Geometric prior $(\msf{C})$. $\msf{D}$ shows the histogram of the data. \label{FigPaw3-3d-MaP}}
%\end{figure}

\begin{figure}[H]
\centering
\includegraphics[scale=0.48]{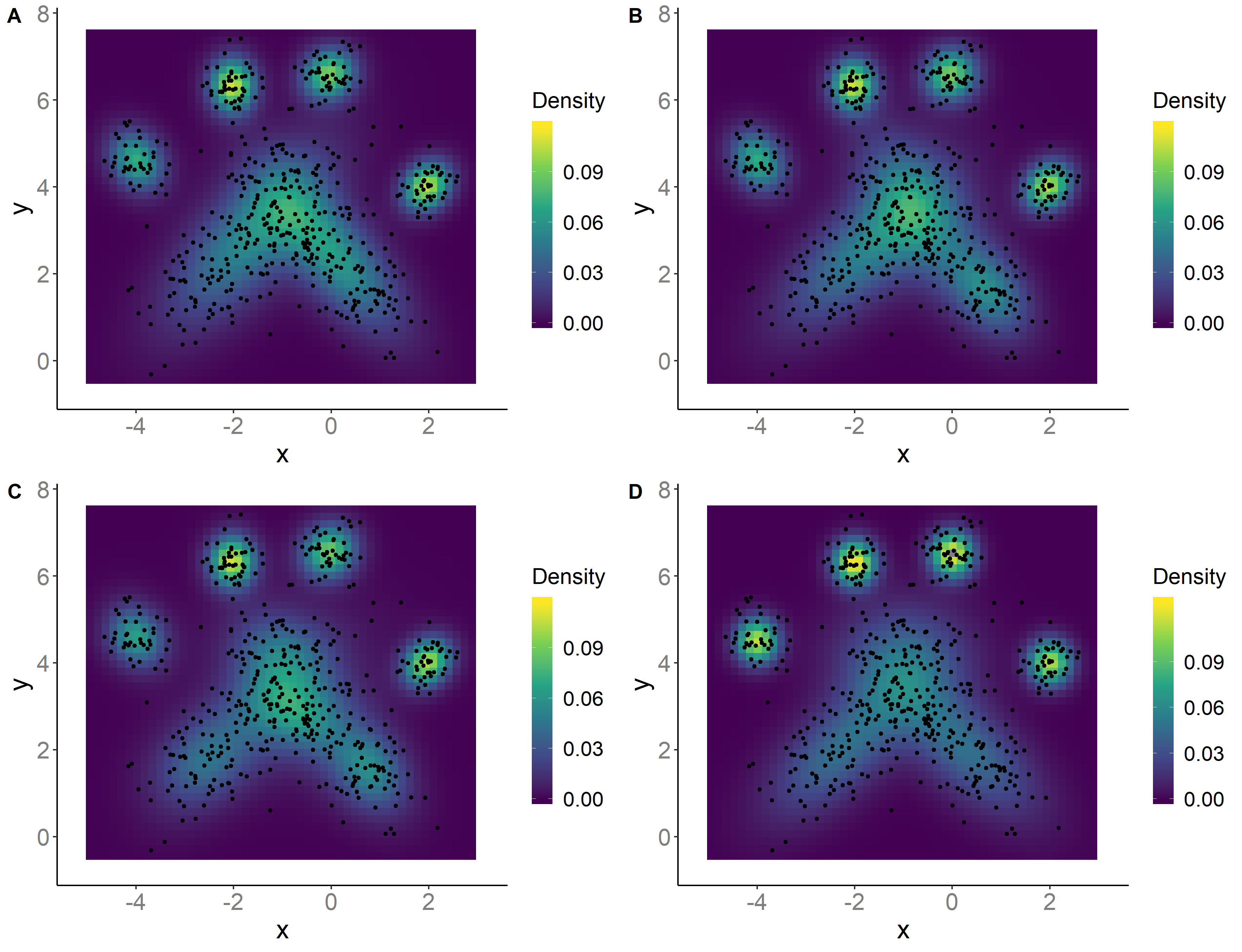} 
\caption{Estimated densities of the data points using the EAP, taking into account each fourth iteration among $8000$ iterations of the Gibbs sampler after a burn-in period of $2000$, according to the Dirichlet prior $(\msf{A})$ a DSB prior $(\msf{B})$ and a Geometric prior $(\msf{C})$. $\msf{D}$ shows the true density from which the data points were i.i.d. sampled. \label{FigPaw3-2d-EaP}}
\end{figure}

Figure \ref{FigPaw3-2d-EaP} (and Figure \ref{FigPaw3-3d-EaP} in the appendix) exhibits the EAP estimators of the density provided by each model, in comparison to the MAP estimators we see that these ones are much smoother. Here we appreciate all models estimate the density quite nicely, this is explained by the fact that Gaussian mixtures are in general very flexible models and that all of the species sampling priors considered here have full support. 

\begin{figure}[H]
\centering
\includegraphics[scale=0.48]{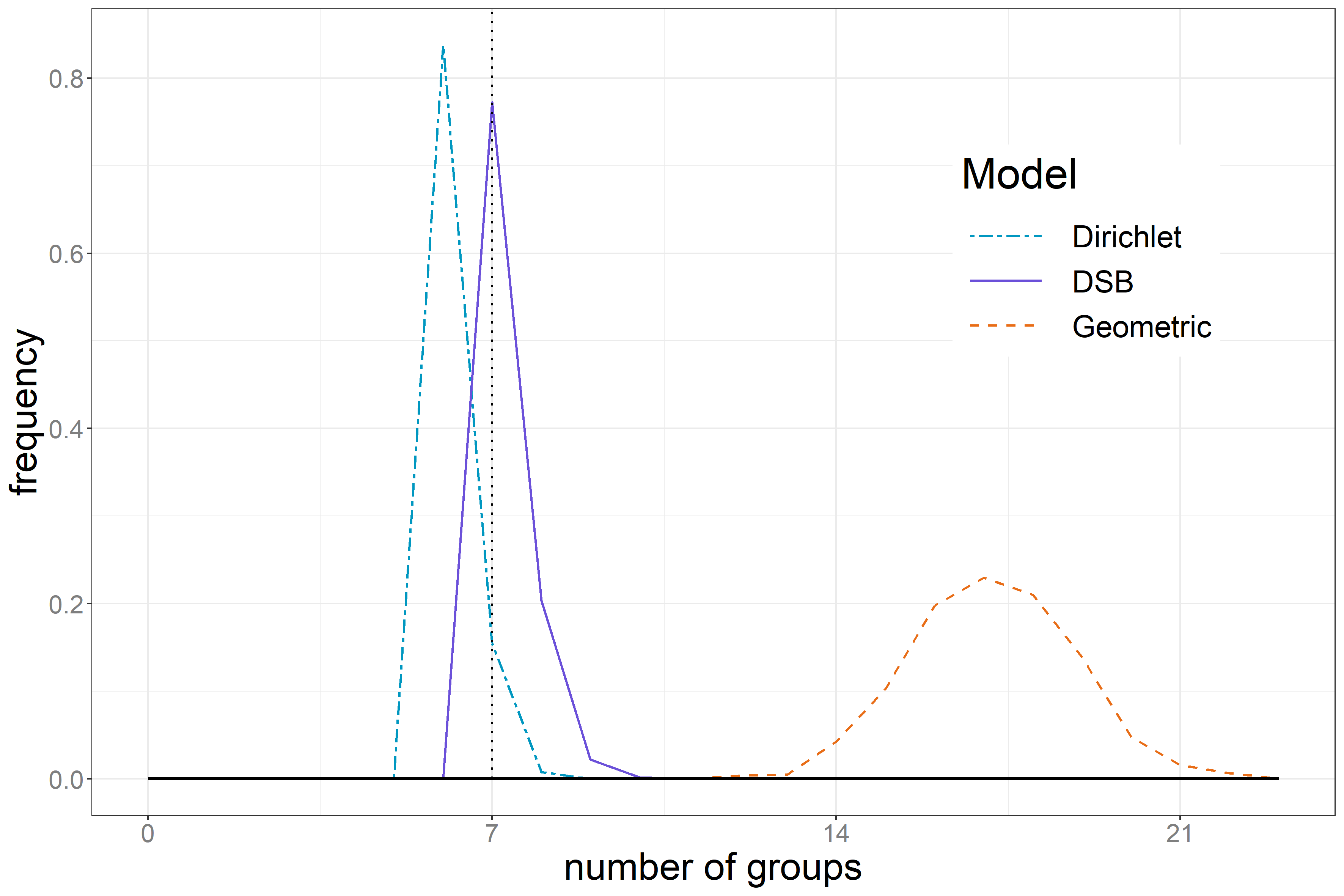} 
\caption{Frequency polygons corresponding to the posterior distribution of $\K_{510}$, for the Dirichlet, DSB, and Geometric mixtures. The dotted line indicates the true number of mixture components. \label{FigPaw3-Kn}}
\end{figure}

In Figure \ref{FigPaw3-Kn} we observe the posterior distribution of $\K_n$, (for $n = 510$) for each of the models. Here we see that through the posterior mode of $\K_n$ the DSB recovers the true number of mixtures components. The Dirichlet model also assigns a probability larger than zero to the true number of components. Despite this, the posterior mode of $\K_n$ for the Dirichlet process is one unit smaller than the true number of components. As to the Geometric process, the posterior distribution of $\K_n$ concentrates in significantly larger values than the real number of mixture components.

The last figure we will analyze is Figure \ref{FigPaw3-tun}, which presents the posterior distribution of $\rho_{\nu}$ for the DSB mixture random random tuning parameter. Here we see that the posterior mode is close to $0.25$. In particular, this suggests that, for the fixed values of the hyper-parameters we considered, a model more similar to the Dirchlet mixture is preferred over one that approximates a Geometric mixture.

%\begin{figure}[H]
%\centering
%\includegraphics[scale=0.85]{paw3_3d_EaP} 
%\caption{$3$D view of the estimated densities using the EAP, taking into account each fourth iteration among $8000$ iterations of the Gibbs sampler after a burn-in period of $2000$, according to the Dirichlet prior $(\msf{A})$ a DSB prior $(\msf{B})$ and a Geometric prior $(\msf{C})$. $\msf{D}$ shows the true density from which the data points were i.i.d. sampled. $\msf{D}$ shows the histogram of the data \label{FigPaw3-3d-EaP}}
%\end{figure}

\begin{figure}[H]
\centering
\includegraphics[scale=0.48]{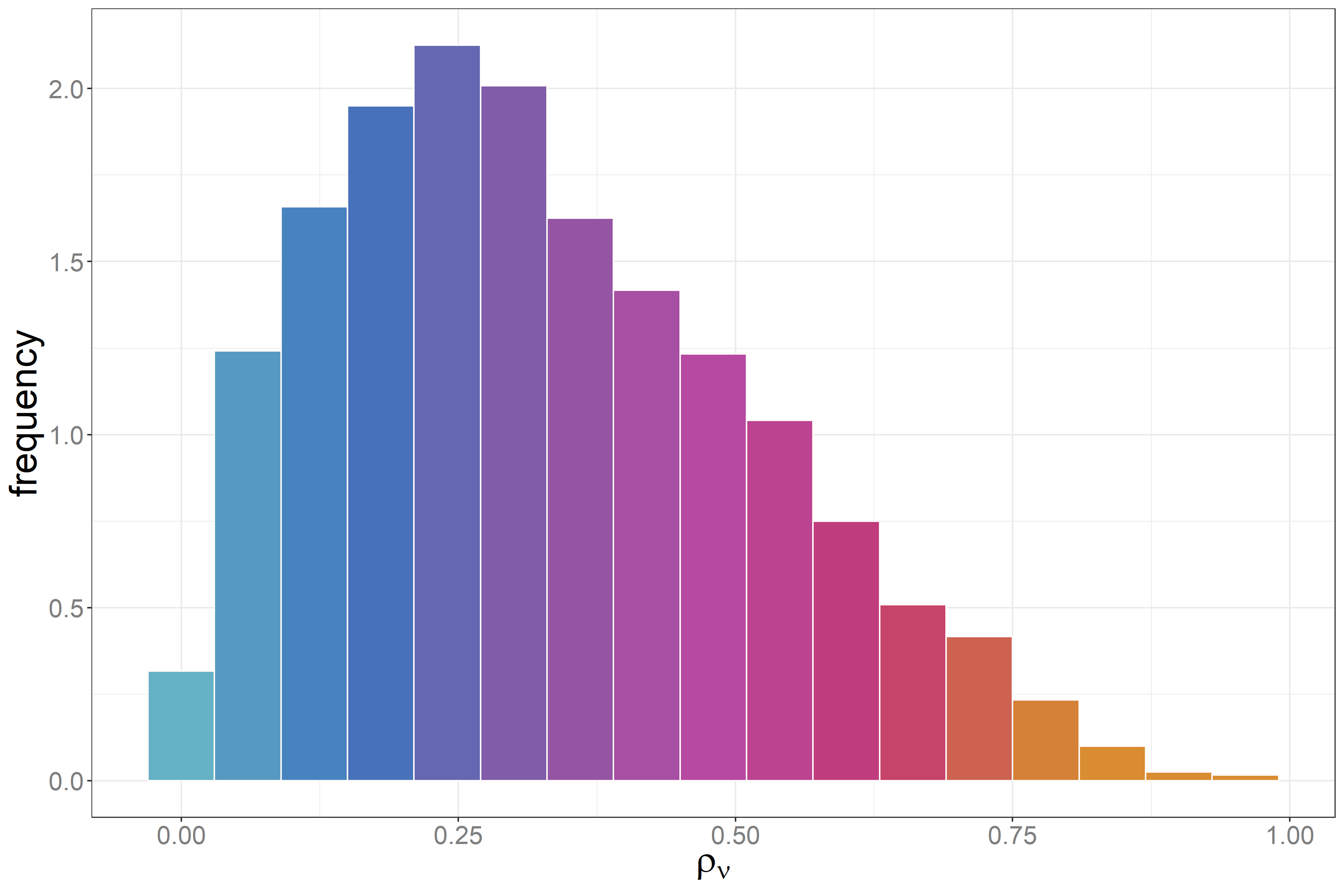} 
\caption{Posterior distributions of the tie probability $\rho_{\nu}$. \label{FigPaw3-tun}}
\end{figure}

\section{Final comments}

{
While we did not addressed in detail other examples of ESBs outside the case where the length variables are species sampling driven, it is important to highlight that various well known Bayesian non-parametric priors fall into the general framework outlined in Section \ref{sec:SBE}. Indeed, if the length variables are conditionally i.i.d. given $\bnu$, where $\bnu$ denotes a parametric distribution with random parameters, for instance $\bnu = \Be(\bm{\alpha},\bm{\theta})$ and $(\bm{\alpha},\bm{\theta}) \sim \bbp(\bm{\alpha},\bm{\theta})$, then the length variables are exchangeable and the species sampling process, $\bmu$, they define is an ESB. Such is the case of a Dirichlet process with a prior on the total mass parameter, model that has been widely exploited \citep[e.g.][]{EW95}. 

Although the present work was mainly motivated by Bayesian non-parametric theory, stick-breaking processes continue to be widely used in other probabilistic frameworks, and the results provided here can easily emigrate to such contexts.
}

%\newpage

\section*{Acknowledgements}
The first author was supported by a CONACyT PhD scholarship. Both authors gratefully acknowledge the support of  PAPIIT Grant IG100221.

\setcounter{figure}{0}
\renewcommand{\thefigure}{A\arabic{figure}}
\setcounter{equation}{0}
\renewcommand{\theequation}{A\arabic{equation}}

\appendix

\section{Stick-breaking decomposition}\label{sec:SB-decomp}

\begin{theo}\label{theo:sb}
Let $\W = (\w_j)_{j \geq 1}$ be a sequence such that $0 \leq \w_j \leq 1$, for every $j \geq 1$, and $\sum_{j\geq 1}\w_j \leq 1$ almost surely. Then, there exist a sequence $\V = (\v_i)_{i \geq 1}$ taking values in $[0,1]$ such that \eqref{eq:sb} holds. 
\end{theo}

\begin{proof}
Let $(\Omega,\cl{F},\Prob)$ be the probability space over which $(\w_j)_{j \geq 1}$ is defined. Given that we are interested in an almost surely decomposition, we may assume without loss of generality that $\sum_{j \geq 1}\w_j \leq 1$ and $0 \leq \w_j \leq 1$, for every $j \geq 1$, hold over $\Omega$. Fix $\v_1 = \w_1$, and for $k \geq 2$, define the event $E_k = \l\{\omega \in \Omega: \sum_{j=1}^{k-1}\w_j(\omega) < 1\r\}$ and set
\[
\v_k = \frac{\w_k}{1-\sum_{j=1}^{k-1}\w_j}\Ind_{E_k}.
\]
Evidently $\v_k$ is measurable as $\w_1,...,\w_k$ are. Also, since $\sum_{j \geq 1}\w_j \leq 1$, we have that $\w_k(\omega) \leq 1-\sum_{j=1}^{k-1}\w_j(\omega)$, for every $\omega \in E_k$, which yields $0 \leq \v_k \leq 1$. This shows $(\v_i)_{i \geq 1}$ is a sequence of $[0,1]$-valued random variables. Now, for $k < k'$ we have that $E_{k'} \subseteq E_{k}$. Hence, for every $k \geq 2$,
\[
\w_k(\omega) = \frac{\w_k(\omega)}{1-\sum_{j=1}^{k-1}\w_j(\omega)}\prod_{j=1}^{k-1}\l(\frac{1-\sum_{i=1}^{j}\w_i(\omega)}{1-\sum_{i=1}^{j-1}\w_i(\omega)}\r) = \v_k(\omega)\prod_{j=1}^{k-1}(1-\v_j(\omega)),
\]
for all $\omega \in E_k$, and since $\sum_{j \geq 1}\w_j \leq 1$, $\w_k(\omega) = 0 = \v_k(\omega)$, for $\omega \in (E_k)^{c}$. This show that for every $k \geq 1$, $\w_k = \v_k\prod_{j=1}^{k-1}(1-\v_j)$, as desired.
\end{proof}

\section{Weak convergence of probability measures}\label{sec:weak_conv}

Inhere we mention some topological details of measure spaces.  For a Polish space $S$, with Borel $\sigma$-algebra $\B_S$, we denote by $\cl{P}(S)$ to the space of all probability measures over $(S,\B_S)$. A well-known metric on $\cl{P}(S)$ is the L\'evy-Prokhorov metric given by
\begin{equation}\label{eq:weak_metric}
d_L(\mu,\mu') = \inf\{\varepsilon > 0: \mu(A) \leq \mu'\l(A^{\varepsilon}\r) + \varepsilon, \mu'(A) \leq \mu\l(A^{\varepsilon}\r)+ \varepsilon, \A A \in \B(S)\},
\end{equation}
for any $\mu,\mu' \in \cl{P}(S)$, and where $A^{\varepsilon} = \{s \in S: d(s,A) < \varepsilon\}$, $d(s,A) = \inf\{d(a,s): a \in A\}$ and $d$ is some complete metric on $S$. For probability measures $\mu,\mu^{(1)},\mu^{(2)},\ldots$ it is said that $\mu^{(n)}$ converges weakly to $\mu$, denoted by $\mu^{(n)} \wto \mu$, whenever $\mu^{(n)}(f) = \int_S f d\mu^{(n)}\to \int_S f d\mu = \mu(f)$ for every continuous bounded function $f:S \to \R$. The Portmanteau theorem states that this condition is equivalent to $d_L\l(\mu^{(n)},\mu\r) \to 0$, and to $\mu^{(n)}(A) \to \mu(A)$, for every Borel set such that $\mu(\partial A) = 0$, where $\partial A$ denotes the boundary of $A$. $\cl{P}(S)$, equipped with the topology of weak convergence, is Polish again. Its Borel $\sigma$-field, $\B_{\cl{P}(S)}$, can equivalently be defined as the $\sigma$-algebra generated by all the projection maps $\{\mu \mapsto \mu(B): B \in \B_S\}$. In this sense the random probability measures, $\l(\bmu^{(n)}\r)_{n \geq 1}$, are said to converge weakly, a.s. to $\bmu$,  whenever $\bmu^{(n)}(\omega) \wto \bmu(\omega)$ outside a $\Prob$-null set. Analogously, we say that $\l(\bmu^{(n)}\r)_{n \geq 1}$ converges weakly, in $\cl{L}_p$, in probability or in distribution, to $\bmu$, denoted by $\bmu^{(n)}\Lpwto \bmu$, $\bmu^{(n)} \Pwto \bmu$ and $\bmu^{(n)} \dwto \bmu$, respectively, whenever $\bmu^{(n)}(f) \to \bmu(f)$, in the corresponding mode of convergence, for every continuous bounded function $f:S \to \R$. Evidently, $\bmu^{(n)} \wto \bmu$ a.s. and $\bmu^{(n)} \Lpwto \bmu$ are both sufficient conditions for $\bmu^{(n)} \Pwto \bmu$, which in turn implies $\bmu^{(n)} \dwto \bmu$. The latter even is equivalent to $\bmu^{(n)} \dto \bmu$ \citep[for further details see][]{P67,B68,K17}. The following Lemmas will be needed  for the proofs of some of the main results.

\begin{lem}\label{lem:dL_cont_map0}
Let $\Delta_{\infty}$ denote the infinite dimensional simplex and consider some Polish space $S$. The mapping
\[
[(\mu_1,\mu_2,...),(w_1,w_2,\ldots)] \mapsto \sum_{j \geq 1}w_j\mu_j,
\]
from $\cl{P}(S)^{\infty} \times \Delta_{\infty}$ into $\cl{P}(S)$ is continuous with respect to the weak and product topologies.
\end{lem}

\textbf{Proof:}
Let $w = (w_1,w_2,\ldots)$, $w^{(n)}= \l(w^{(n)}_1,w^{(n)}_2,\ldots\r)_{n \geq 1}$ be elements of $\Delta_{\infty}$, and  $\mu = (\mu_1,\mu_2,\ldots)$, $\mu^{(n)}=\l(\mu^{(n)}_1,\mu^{(n)}_2,\ldots\r)_{n \geq 1}$, be elements of $\cl{P}(S)^{\infty}$, such that $w_j^{(n)} \to w_j$ and  $\mu^{(n)}_j \wto \mu_j$, for every $j \geq 1$. Define $\nu^{(n)} = \sum_{j \geq 1}w^{(n)}_j\mu^{(n)}_j$ and $\nu = \sum_{j \geq 1}w_j\mu_j$. Fix a continuous and bounded function $f: S \to \R$. Then, for $j \geq 1$, $w^{(n)}_j\mu^{(n)}_j(f) \to w_j\mu_j(f)$. Since $f$ is bounded, there exist $M$ such that $|f|\leq M$, hence $|w^{(n)}_j\mu^{(n)}_j(f)| \leq w^{(n)}_j\mu^{(n)}_j(|f|)\leq w^{(n)}_jM$, for every $n \geq 1$, and $j \geq 1$. Evidently, $Mw^{(n)}_j \to Mw_j$, and $\sum_{j \geq 1}Mw^{(n)}_j = M = \sum_{j \geq 1}Mw_j$. Hence, by general Lebesgue dominated convergence theorem, we obtain
\[
\nu^{(n)}(f) = \sum_{j\geq 1}w^{(n)}_j\mu^{(n)}_j(f) \to  \sum_{j\geq 1}w_j\mu_j(f) = \nu(f)
\]
That is $\nu^{(n)} \wto \nu$.
\qed

\begin{lem}\label{lem:dL_cont_map}
Let $\Delta_{\infty}$ denote the infinite dimensional simplex and consider some Polish space $S$. The mapping
\[
[(s_1,s_2,...),(w_1,w_2,\ldots)] \mapsto \sum_{j \geq 1}w_j\delta_{s_j},
\]
from $S^{\infty} \times \Delta_{\infty}$ into $\cl{P}(S)$ is continuous with respect to the product and weak topologies.
\end{lem}

\textbf{Proof:}
By Lemma \ref{lem:dL_cont_map0} if suffices to check that the mapping $s \to \delta_s$ from $S$ into $\cl{P}(S)$ is continuous. So fix $s_{n} \to s$ in $S$ and fix a continuous and bounded function $f:S\to \R$. Then $\delta_{s_{n}}(f) = f(s_{n}) \to f(s) = \delta_s(f)$, that is $\delta_{s_{n}} \wto \delta_s$.
\qed

\begin{lem}\label{lem:dL_cont_map_1}
Consider two Polish spaces, $S$ and $T$, and let $G$ be a probability kernel from $S$ into $T$, such that for every $s_{n} \to s$ in $S$, $G(\cdot|s_{n}) \wto G(\cdot|s)$. The mapping
\[
\mu = \sum_{j \geq 1}w_j\delta_{s_j} \mapsto \int G(\cdot\mid s)d\mu(ds) = \sum_{j \geq 1}w_jG(\cdot \mid s_j),
\]
from $\cl{P}(S)$ into $\cl{P}(T)$ is continuous with respect to the product and weak topologies.
\end{lem}

\textbf{Proof:} 
Consider some discrete probability measures $\l(\mu^{(n)} = \sum_{j \geq 1}w^{(n)}_j\delta_{s^{(n)}_j}\r)_{n \geq 1}$ over $(S,\B_S)$, such that $\mu^{(n)} \wto \sum_{j \geq 1}w_j\delta_{s_j} = \mu$, and set $\Phi^{(n)} = \sum_{j \geq 1}w^{(n)}_jG\l(\cdot\mi s^{(n)}_j\r)$ and $\Phi = \sum_{j \geq 1}w_jG\l(\cdot\mid s_j\r)$. Let $f:T \to \R$ be a continuous and bounded function and define the function $h:S \to \R$, by
\[
h(s) = \int f(t) G(dt\mid s).
\]
Evidently $h$ is bounded because $f$ is bounded and $G(\cdot|s)$ is a probability measure. Furthermore, as $G(\cdot|s_{n}) \wto G(\cdot|s)$, for every $s_{n} \to s$ in $S$, $h$ is also continuous. Thus,
\[
\int f \, d\Phi^{(n)} = \sum_{j \geq 1} w^{(n)}_j h\l(s^{(n)}_j\r)= \int h\, d\mu^{(n)} \to \int h \,d\mu = \sum_{j \geq 1} w_j h\l(s_j\r) = \int f \, d\Phi.
\]
That is $\Phi^{(n)} \wto \Phi$.
\qed

\begin{cor}\label{cor:dL_cont_map2}
Let $\Delta_{\infty}$ denote the infinite dimensional simplex. Consider a couple of Polish spaces, $S$ and $T$ and let $G$ be a probability kernel from $S$ into $T$, such that for every $s_{n} \to s$ in $S$, $G(\cdot|s_{n}) \wto G(\cdot|s)$. The mapping
\[
[(s_1,s_2,...),(w_1,w_2,\ldots)] \mapsto \sum_{j \geq 1}w_jG(\cdot|s_j),
\]
from $S^{\infty} \times \Delta_{\infty}$ into $\cl{P}(T)$ is continuous with respect to the weak topology.
\end{cor}

\textbf{Proof:} This is straightforward from Lemmas \ref{lem:dL_cont_map} and \ref{lem:dL_cont_map_1}.

\section{Exchangeable sequences driven by species sampling processes}\label{sec:exch_ssp_supp}

This section provides an overview of exchangeable sequences driven by species sampling processes. The results presented here are fundamental for the proof of our main results.

\begin{theo}\label{theo:Rep_SSP}
Let $(\x_i)_{i \geq 1}$ be an random sequence, taking values in a Polish space $(S,\B(S))$, and for $n \geq 1$, define $\bPi(\x_{1:n})$ as the random partition of $[n]$ generated by the random equivalence relation $i \bm{\sim} j$ if and only if $\x_i = \x_j$. Let $\mu_0$ be a diffuse probability measure over $(S,\B(S))$ and let $\pi$ be an EPPF. The following statements are equivalent.
\begin{itemize}
\item[\emph{I.}]  $(\x_i)_{i \geq 1}$ is exchangeable and directed by a species sampling process $\bmu$ as in \eqref{eq:ssp}, with base measure $\mu_0$, and whose size-biased pseudo-permuted weights $(\tw_j)_{j \geq 1}$ satisfy \[\pi(n_1,\ldots,n_k) = \Esp\l[\prod_{j=1}^k\tw_j^{n_j-1}\prod_{j=1}^{k-1}\l(1-\sum_{i=1}^j \tw_j\r)\r].\]
\item[\emph{II.}] $\x_1 \sim \mu_0$, and for every $n \geq 1$,
\[
\Prob[\x_{n+1} \in \cdot\mid \x_1,\ldots,\x_{n}] = \sum_{j=1}^{\K_n}\frac{\pi\l(\n^{(j)}\r)}{\pi(\n)}\delta_{\x^*_j} + \frac{\pi\l(\n^{(\K_n+1)}\r)}{\pi(\n)}\mu_0,
\]
where $\x^*_1,\ldots,\x^{*}_{\K_n}$ are the $\K_n$ distinct values in $\{\x_1,\ldots,\x_n\}$, $\n = (\n_1,\ldots\n_{\K_n})$,  $\n^{(j)} = (\n_1,\ldots \n_{j-1},\n_j+1,\n_{j+1}, \ldots,\n_{\K_n})$ and  $\n^{(\K_n+1)} = (\n_1,\ldots,\n_{\K_n},1)$,  with $\n_{j} = |\{i\leq n:\x_i = \x^{*}_j\}|$.
\item[\emph{III.}] The law of $(\bPi(\x_{1:n}))_{n \geq 1}$ is described by the EPPF $\pi$ and for every $n \geq 1$ and $B_1,\ldots,B_n \in \B(S)$
\[
\Prob\l[\x_1\in B_1,\ldots,\x_n \in B_n \mid \bPi(\x_{1:n})\r] = \prod_{i=1}^{\K_n}\mu_0\l(\bigcap_{j\in \bPi_i}B_j\r),
\]
where $\bPi_1,\ldots,\bPi_{\K_n}$ are the random blocks of $\bPi(\x_{1:n})$.
\item[\emph{IV.}] For every $n \geq 1$, and any $x_1,\ldots,x_n \in S$,
\[
\Prob\l[\x_1 \in dx_1,\ldots\x_n \in dx_n\r] = \pi(n_1,\ldots,n_k) \prod_{i=1}^k \mu_0(dx^{*}_j),
\]
where $x^*_1,\ldots,x^{*}_{k}$ are the distinct values in $\{x_1,\ldots,x_n\}$, and $n_{j} = |\{i:x_i = x^{*}_j\}|$.
\end{itemize}
\end{theo}

First we clarify what we mean by a size-biased pseudo-permutation. For a sequence of weights, $\W = (\w_j)_{j \geq 1}$, with $\w_j \geq 0$ and $\sum_{j \geq 1} \w_j \leq 1$ almost surely, we call $\tW = \l(\tw_1,\tw_2,\ldots\r)$, a size-biased pseudo-permutation of $\W$ if 
\[
\Prob\l[\tw_1 \in \cdot\mi\W\r] = \sum_{j \geq 1}\w_j\delta_{\w_j}+\l(1-\sum_{j \geq 1}\w_j\r)\delta_0,
\]
and for every $i \geq 1$,
\begin{align*}
&\Prob\l[\tw_{i+1} \in \cdot\mid\W,\tw_1,\ldots,\tw_i\r] = \frac{\sum_{j \geq 1}\w_j\delta_{\w_j}-\sum_{j=1}^{i}\tw_j\delta_{\tw_j}+\l(1-\sum_{j \geq 1}\w_j\r)\delta_0}{\l(1-\sum_{j=1}^i\tw_j\r)}
\end{align*}
if $\l(1-\sum_{j=1}^i\tw_j\r) > 0$, and $\Prob\l[\tw_{i+1} \in \cdot \mi \W,\tw_1,\ldots\tw_i\r] = \delta_0$, otherwise. If the weights sum up to one almost surely, this definition coincides with the notion of size-biased permutation of the weights \citep{P96b}. 

The proof of Theorem \ref{theo:Rep_SSP} is a consequence of the work by \cite{P95,P96b,P96}. In particular, point III reveals that for $\{\x_i\mid \bmu \iid \bmu ; \, i \geq 1\}$ driven by a species sampling process, the random vector $(\x_1,\ldots,\x_n)$ is distributed as $\l(\x^{*}_{l_1},\ldots,\x^{*}_{l_n}\r)$, with $l_r = j$ if and only if $r \in \bPi_j$. For example, say that for some realization $\bPi(\x_{1:6}) = \{\{1,4,5\},\{2,3\},\{6\}\}$, then under such event, $(\x_1,\ldots,\x_6)$ distributes as $(\x^*_1,\x^{*}_2,\x^{*}_2,\x^{*}_1,\x^*_1,\x^*_3)$, where $\{\x^*_i \iid \mu_0;\, i \geq 1\}$ independently of $\bPi(\x_{1:6})$. With this in mind, the proof of the following result is straightforward. 

\begin{theo}\label{theo:means_samples}
Let $(\x_i)_{i \geq 1}$ be an exchangeable sequence driven by a species sampling process, $\bmu$, with base measure $\mu_0$ and corresponding EPPF $\pi$. Fix $n \geq 1$ and let $f:S^{n} \to \R$ be measurable function, then
%\begin{equation}
\begin{align}
&\Esp\l[f(\x_1,\ldots,\x_n)\r]\label{eq:means_samples_supp}\\
& = \sum_{\{A_1,\ldots,A_k\}}\l\{\int f(x_{l_1},\ldots,x_{l_n})\prod_{j=1}^{k}\prod_{r \in A_j}\Ind_{\{l_r = j\}}\,\mu_0(dx_1)\ldots\mu_0(dx_k)\r\}\pi(|A_1|,\ldots,|A_k|),\nonumber
\end{align}
%\end{equation}
whenever the integrals in the right side exist, and where the sum ranges over all partitions of $\{1,\ldots,n\}$. 
\end{theo}

\begin{proof}
Let $\bPi_1,\ldots,\bPi_{\K_n}$ denote the blocks of $\bPi(\x_{1:n})$, and consider $\{\x^*_i \iid \mu_0;\, i \geq 1\}$ independently. By Theorem \ref{theo:Rep_SSP} III, and the tower property of conditional 
\begin{align*}
&\Esp\l[f(\x_1,\ldots,\x_n)\r] = \Esp\l[\Esp\l[f(\x_1,\ldots,\x_n)\mid \bPi(\x_{1:n})\r]\r]\\
& = \Esp\l[f(\x^*_{l_1},\ldots,\x^{*}_{l_n})\prod_{j=1}^{\K_n}\prod_{r \in \bPi_j}\Ind_{\{l_r = j\}}\r]\\
& = \sum_{\{A_1,\ldots,A_k\}}\l\{\int f(x_{l_1},\ldots,x_{l_n})\prod_{j=1}^{k}\prod_{r \in A_j}\Ind_{\{l_r = j\}}\,\mu_0(dx_1)\ldots\mu_0(dx_k)\r\}\pi(|A_1|,\ldots,|A_k|),
\end{align*}
whenever the integral in the right side exist, and where the sum ranges over all partitions of $\{1,\ldots,n\}$.
\end{proof}

Theorem \ref{theo:means_samples} generalizes a result by \cite{Y84}, in which \eqref{eq:means_samples_supp} is derived only for symmetric functions and for the special case where $\bmu$ is a Dirichlet process. Related formulae also appear and are cleverly exploited in \cite{LP09}. 

As mention in Section \ref{sec:exch_ssp}, the tie probability, $\rho = \Prob[\x_1 = \x_2] = \pi(2)$, determines important characteristics of a species sampling process. For instance the following conditional moments are completely determined by the tie probability and the base measure.

\begin{cor}\label{cor:xi_xj_mom}
Let $\bmu$ be a species sampling process with base measure $\mu_0$ and tie probability $\rho$. Consider $\{\x_1,\x_2,\ldots\mid \bmu\}\sim \bmu$. Then, for every $i \neq j$
\begin{itemize}
\item[\emph{a)}] $\Esp[\x_j\mid\x_i] = \rho\, \x_i + (1-\rho)\Esp[\x_i]$
\item[\emph{b)}] $\Var(\x_j\mid\x_i) = (1-\rho)\l\{\rho\l(\x_i-\Esp[\x_i]\r)^2 + \Var(\x_i)\r\}$
\item[\emph{c)}] $\Cov(\x_i,\x_j) = \rho\,\Var(\x_i)$
\item[\emph{d)}] $\Corr(\x_i,\x_j) = \rho$.
\end{itemize}
\end{cor}

\begin{proof}
For any measurable function $f:[0,1] \to \R_+$, and for every $i \neq j$,
\[
\Esp\l[f(\x_j)\mi\x_i\r] = \rho\,f(\x_i) +(1-\rho)\int f(x) \mu_0(dx) = \rho\,f(\x_i) +(1-\rho)\Esp\l[f(\x_i)\r],
\]
noting that $\x_i \sim \mu_0$. The choice $f(x) = x$ proves (a). To prove (b) note that for $f(x) = x^2$, we obtain $\Esp\l[\x_j^2\mi\x_i\r] =\rho\, \x_i^2 +(1-\rho)\Esp\l[\x_i^{2}\r]$, this together with (a) show that
\begin{align*}
\Var(\x_j\mid \x_i) & = \rho\,\x_i^2+(1-\rho)\Esp[\x_i^2]-\l(\rho\,\x_i+(1-\rho)\Esp[\x_i]\r)^2\\
& = (1-\rho)\l\{\rho\l(\x_i-\Esp[\x_i]\r)^2+\Var(\x_i)\r\}.
\end{align*}
To prove (c) we first compute, using (a), $\Esp\l[\x_i\x_j\r] = \Esp[\x_i\Esp\l[\x_j\mi\x_i\r]] = \rho\Esp\l[\x_i^2\r] + (1-\rho)\Esp[\x_i]^2$. Thus
\[
\Cov(\x_i,\x_j) = \rho\Esp\l[\x_i^2\r] + (1-\rho)\Esp[\x_i]^2 - \Esp[\x_i]^2 = \rho\,\Var(\x_i).
\]
Finally, (d) follows by diving the last equation by $\Var(\x_i) = \sqrt{\Var(\x_i)\Var(\x_j)}$.
\end{proof}

Notice that for small values of $\rho$, $\Esp[\x_j\mid \x_i] \approx \Esp[\x_j]$, $\Var(\x_j\mid\x_i)  \approx \Var(\x_j)$ and $\Cov(\x_i,\x_j) \approx 0$, alternatively for values of $\rho $ close to $1$,  $\Esp[\x_j\mid \x_i] \approx \x_i$, $\Var(\x_j\mid\x_i)  \approx 0$ and $\Cov(\x_i,\x_j) \approx \Var(\x_i)$. So  $(\x_i)_{i \geq 1}$ behaves very similar to an i.i.d. sequence, whenever $\rho$ is close to $0$, and if $\rho \approx 1$ the behaviour of $(\x_i)_{i \geq 1}$ is similar to $(\x,\x,\ldots)$  with $\x \sim \mu_0$. Theorem \ref{theo:rho_limit} below formalizes this intuition, to prove it we use the following lemma, which also allows to characterize some moments of the species sampling process  in terms of its base measure and its tie probability.

\begin{lem}\label{lem:Emu(f)}
Let $S$ be a Polish space and consider a species sampling process $\bmu$ as in \eqref{eq:ssp}, with base measure $\mu_0$ and with tie probability $\rho$. Let $f,g:S \to \R$ be measurable and bounded functions. Let us denote $\bmu(f) = \int f(s) \bmu(ds)$ and analogously for $\mu_0$ and $g$. Then,
\begin{itemize}
\item[\emph{a)}] $\Esp\l[\bmu(f)\r] = \mu_0(f)$.
\item[\emph{b)}] $\Esp\l[\bmu(f)^2\r] = \rho \,\mu_0(f^2) + (1-\rho)\mu_0(f)^2$.
\item[\emph{c)}] $\Esp\l[\bmu(f)\bmu(g)\r] = \rho\, \mu_0(fg) + (1-\rho)\mu_0(f)\mu_0(g)$.
\end{itemize}
\end{lem}

\begin{proof}
Set $\w_0 = 1-\sum_{j \geq 1}\w_j$, so that $\sum_{j \geq 0}\w_j  = 1$ almost surely, and by a monotone convergence argument we also obtain $\sum_{j \geq 0}\Esp\l[\w_j\r] = 1$. Note that 
\[
\bmu(f) = \sum_{j \geq 1}\w_j f(\bxi_j) + \w_0\mu_0(f),
\]
and for any bound of $f$, $M$, we have that $|\sum_{j=1}^n\w_jf(\bxi_j)| < M$ almost surely for every $n \geq 1$. Hence, by linearity of the expectation, Lebesgue dominated convergence theorem, and since $\{\bxi_j \iid \mu_0;\, j \geq 1\}$ independently of the weights, we get
\begin{align*}
\Esp[\bmu(f)]&  = \sum_{j \geq 1}\Esp[\w_j]\Esp[f(\bxi_j)] + \Esp\l[\w_0\r]\mu_0(f) = \l(\sum_{j \geq 0}\Esp\l[\w_j\r]\r)\mu_0(f) = \mu_0(f).
\end{align*}
This proves the first part. To prove the second and thirds parts, first note that by the tower property of conditional expectation and monotone convergence theorem we get $\rho = \sum_{j \geq 1}\Esp\l[\l(\w_j\r)^2\r]$ and
\[
1 = \Esp\l[\l(\sum_{j \geq 0}\w_j\r)^2\r] = \sum_{j \geq 1}\Esp\l[\w_j^2 \r]+ \sum_{i \neq j}\Esp\l[\w_i\w_j\r] + \Esp\l[\w_0^{2}\r].
\]
Thus, $1-\rho = \sum_{i \neq j}\Esp\l[\w_i\w_j\r] + \Esp\l[\w_0^{2}\r]$, where $\sum_{i \neq j}a_ia_j$ denotes $\sum_{i \geq 0}\sum_{j \geq 0} a_ia_j\Ind_{\{i \neq j\}}$. Secondly, since $f$ and $g$ are bounded and $\bmu$ is a random probability measure we have that $\bmu(|f|),\bmu(|g|) < \infty$. Then,
\begin{align*}
\bmu(f)\bmu(g) & = \l(\sum_{j \geq 1}\w_j f(\bxi_j) + \w_0\mu_0(f)\r)\l(\sum_{j \geq 1}\w_j g(\bxi_j) + \w_0\mu_0(g)\r)\\
& = \sum_{j \geq 1}\w_j^2f(\bxi_j)g(\bxi_j) + \sum_{\stackrel{i,j \geq 1}{i \neq j}}\w_i\w_jf(\bxi_i)g(\bxi_j)+ \l(\sum_{j \geq 1}\w_0\w_jg(\bxi_j)\r)\mu_0(f)\\
& \quad \quad \quad \quad \quad \quad \quad \quad \quad \quad \quad \quad + \l(\sum_{j \geq 1}\w_0\w_jf(\bxi_j)\r)\mu_0(g) + \w_0^2\mu_0(f)\mu_0(g).
\end{align*}
Now, if $M$ is a bound for $f$, and $N$ is a bound of $g$ we have that for every $n \geq 1$, $|\sum_{j=1}^n\w_j\w_0f(\bxi_j))| \leq M$, $|\sum_{j=1}^n\w_j\w_0g(\bxi_j))| \leq N$, $|\sum_{j=1}^n\w_j^2f(\bxi_j)g(\bxi_j)| \leq MN$, and $|\sum_{i=1}^n\sum_{j=1}^n\w_i\w_jf(\bxi_j)g(\bxi_i)\Ind_{\{i \neq j\}}| \leq MN$. Thus, by linearity of the expectation, Lebesgue dominated convergence theorem, and as $\{\bxi_j \iid \mu_0;\, j \geq 1\}$ independently of the weights, we obtain
\begin{align*}
&\Esp\l[\bmu(f)\bmu(g)\r]\\
&= \sum_{j \geq 1}\Esp\l[\w_j^2\r]\Esp\l[f(\bxi_j)g(\bxi_j)\r] + \sum_{\stackrel{i,j \geq 1}{i \neq j}}\Esp\l[\w_i\w_j\r]\Esp\l[f(\bxi_i)\r]\Esp\l[g(\bxi_j)\r]+\Esp\l[\w_0^2\r]\mu_0(f)\mu_0(g)\\
& \quad \quad \quad \quad \quad \quad + \l(\sum_{j \geq 1}\Esp\l[\w_0\w_j\r]\Esp\l[g(\bxi_j)\r]\r)\mu_0(f)+ \l(\sum_{j \geq 1}\Esp\l[\w_0\w_j\r]\Esp\l[f(\bxi_j)\r]\r)\mu_0(g).\\
& = \sum_{j \geq 1}\Esp\l[\w_j^2\r]\mu_0(fg) + \sum_{i \neq j}\Esp\l[\w_i\w_j\r]\mu_0(f)\mu_0(g) + \Esp\l[\w_0^2\r]\mu_0(f)\mu_0(g)\\
& = \rho\mu_0(fg)+(1-\rho)\mu_0(f)\mu_0(g).
\end{align*}
This proves the third part of the lemma, and the choice $g = f$ gives the second part.
\end{proof}

In the context of Lemma \ref{lem:Emu(f)}, the particular choices $f = \Ind_A$ and $g = \Ind_{B}$ for some $A,B \in \B_S$, imply $\Esp[\bmu(A)] = \mu_0(A)$, $\Var\l(\bmu(A)\r) = \rho\,\mu_0(A)(1-\mu_0(A))$ and $\Cov\l(\bmu(A),\bmu(B)\r) = \rho (\mu_0(A\cap B) - \mu_0(A)\mu_0(B))$.

\begin{theo}\label{theo:rho_limit}
Consider a Polish space $S$ with Borel $\sigma$-algebra $\B_S$. Let $\mu_0,\mu^{(1)}_0,\mu^{(2)}_0,\ldots$ be diffuse probability measures over $(S,\B_S)$, such that $\mu^{(n)}_0$ converges weakly to $\mu_0$ as $n \to \infty$. For $n \geq 1$ let $\rho^{(n)} \in (0,1)$, and consider $\l\{\x_i^{(n)}\mi \bmu^{(n)} \iid \bmu^{(n)} ; \, i \geq 1\r\}$ where $\bmu^{(n)}$ is a species sampling process with base measure $\mu^{(n)}_0$ and tie probability $\rho^{(n)}$. 
\begin{itemize}
\item[\emph{i)}] If $\rho^{(n)} \to 0$, as $n \to \infty$, then $\bmu^{(n)}$ converges weakly in distribution to $\mu_0$, and $\l(\x^{(n)}_i\r)_{i \geq 1}$ converge in distribution to a sequence of i.i.d. random variables $\{\x_i \iid \mu_0;\, i \geq 1\}$.
\item[\emph{ii)}] If $\rho^{(n)} \to 1$, as $n \to \infty$, then $\bmu^{(n)}$ converges weakly in distribution to $\delta_{\x}$, where $\x \sim \mu_0$, and $\l(\x^{(n)}_i\r)_{i \geq 1}$ converge in distribution to the sequence of identical random variables $(\x,\x,\ldots)$.
\end{itemize}
\end{theo}

\begin{proof}
We may assume without loss of generality that all the species sampling processes are defined on the same probability space. First we prove (i), let $f: S \to \R$ be a continuous and bounded function. Since $f$ is continuous it is also measurable, and by Lemma \ref{lem:Emu(f)} we have that
\begin{equation}\label{eq:muL2}
\begin{aligned}
\Esp&\l[\l\{\bmu^{(n)}(f)-\mu_0(f)\r\}^2\r]\\
& = \Esp\l[\l\{\bmu^{(n)}(f)\r\}^2\r]- 2\Esp\l[\bmu^{(n)}(f)\r]\mu_0(f) + \l\{\mu_0(f)\r\}^2\\
& =  \rho^{(n)}\, \mu^{(n)}_0\l(f^2\r) + \l(1-\rho^{(n)}\r)\l\{\mu^{(n)}_0(f)\r\}^2 - 2\mu^{(n)}_0(f)\mu_0(f) + \l\{\mu_0(f)\r\}^2.
\end{aligned}
\end{equation}
By hypothesis we know that $\mu^{(n)}_0 \wto \mu_0$ and $\rho^{(n)} \to 0$, as $n \to \infty$, by taking limits in \eqref{eq:muL2}, we found that 
\[
\Esp\l[\l\{\bmu^{(n)}(f)-\mu_0(f)\r\}^2\r] \to 0,
\]
as $n \to \infty$. That is, $\bmu^{(n)}(f)$ converges to $\mu_0(f)$ in $\cl{L}_2$, which implies $\bmu^{(n)}(f) \dto \mu_0(f)$. Since  $f$ was chosen arbitrarily, this proves (i) for the species sampling processes. Given that $S$ and $\cl{P}(S)$ are Polish, we might construct on some probability space $\l(\hat{\Omega},\hat{\F},\hat{\Prob}\r)$ some exchangeable sequences $\l\{\hat{\X}^{(n)} = \l(\hat{\x}^{(n)}_i\r)_{i \geq 1}\r\}_{n \geq 1}$, such that $\hat{\X}^{(n)}$ is directed by a species sampling process, $\hat{\bmu}^{(n)}$, with base measure $\mu_0^{(n)}$ and tie probability $\rho^{(n)}$, and where $\hat{\bmu}^{(n)}$ converges weakly almost surely to $\mu_0$, as $n \to \infty$. Fix $m \geq 1$ and $B_1,\ldots,B_m \in \B_S$. Since $\mu_0$ is diffuse, $\mu_0(\partial B_i) = 0$, and by the Portmanteau theorem we know $\hat{\bmu}^{(n)}(B_i) \to \mu_0(B_i)$ almost surely as $n \to \infty$. This together with the representation theorem for exchangeable sequences imply
\[
\hat{\Prob}\l[\bigcap_{i=1}^m \l(\hat{\x}^{(n)}_i \in B_i\r)\mi \hat{\bmu}^{(n)}\r] = \prod_{i=1}^{m}\hat{\bmu}^{(n)}(B_i) \to \prod_{i=1}^{m}\mu_0(B_i),
\]
almost surely, as $n \to \infty$, and by taking expectations we obtain
\[
\hat{\Prob}\l[\bigcap_{i=1}^m \l(\hat{\x}^{(n)}_i \in B_i\r)\r] \to \hat{\Esp}\l[\prod_{i=1}^{m}\mu_0(B_i)\r] = \prod_{i=1}^{m}\mu_0(B_i) = \hat{\Prob}\l[\bigcap_{i=1}^m \l(\hat{\x}_i \in B_i\r)\r],
\]
where $\hat{\X} = \l(\hat{\x}_i\r)_{i \geq 1}$, with $\{\hat{\x}_i \iid \mu_0;\, i \geq 1\}$. Thus $\l(\x^{(n)}_i\r)_{i \geq 1} \deq \hat{\X}^{(n)} \dto \hat{\X}$ as $n \to \infty$, and we have proven (i).

To prove (ii)  let $\w^{(n)}_1 \geq \w^{(n)}_2 \geq \cdots$ be the decreasingly ordered weights of $\bmu^{(n)}$ and let $\bxi^{(n)}_j$ be the atom corresponding to $\w^{(n)}_j$. Let us denote $\w^{(n)}_0 = 1-\sum_{j \geq 1}\w^{(n)}_j$. Note that, using monotone convergence theorem, we can write
\begin{equation}\label{eq:rho_n}
\rho^{(n)} = \Esp\l[\Prob\l[\x^{(n)}_1 = \x^{(n)}_2 \mi \bmu^{(n)}\r]\r] = \sum_{j \geq 1}\Esp\l[\l(\w^{(n)}_j\r)^2\r]
\end{equation}
and by the proof of Lemma \ref{lem:Emu(f)} we also know
\begin{equation}\label{eq:1-rho_n}
1-\rho^{(n)} = \Esp\l[\l(\w^{(n)}_0\r)^2\r] + \sum_{i \neq j}\Esp\l[\w^{(n)}_j\w^{(n)}_i\r]
\end{equation}
for $n \geq 1$. Since the weights are decreasing, we must have that for every $i \geq j \geq 2$, $\Esp\l[\w^{(n)}_i\w^{(n)}_j\r] \leq \Esp\l[\w^{(n)}_i\w^{(n)}_{j-1}\r]$, hence
\begin{equation}\label{eq:sum_wij_bound}
\sum_{i \neq j}\Esp\l[\w^{(n)}_i\w^{(n)}_j\r] \geq \sum_{j \geq 2}\sum_{i \geq j}\Esp\l[\w^{(n)}_i\w^{(n)}_j\r] \geq \sum_{j \geq 2}\Esp\l[\l(\w^{(n)}_j\r)^2\r] \geq 0,
\end{equation}
for $n \geq 1$. By taking limits, as $n \to \infty$, by \eqref{eq:1-rho_n} and \eqref{eq:sum_wij_bound}, we found $\sum_{j \geq 2}\Esp\l[\l(\w^{(n)}_j\r)^2\r] \to 0$, which together with \eqref{eq:rho_n} proves that $\Esp\l[\l(\w^{(n)}_1\r)^2\r] \to 1$. Since $0 \leq \w^{(n)}_1 \leq 1$, and $\sum_{j \geq 0} \Esp\l[\w^{(n)}_j\r]= 1$, we obtain
\begin{equation}\label{eq:Ew1Ew0}
\Esp\l[\w^{(n)}_1\r] \to 1 \quad  \text{ and } \quad \sum_{j \neq 1}\Esp\l[\w^{(n)}_j\r] \to 0,
\end{equation}
as $n \to \infty$. Seeing that  all the corresponding spaces are Polish, and $\mu^{(n)}_0 \wto \mu_0$, we might construct on a probability space $\l(\hat{\Omega},\hat{\F},\hat{\Prob}\r)$, some independent sequences, $\l\{\hat{\bxi}_j^{(n)} \iid \mu_0^{(n)};\, j \geq 1\r\}$, and $\l(\hat{\w}^{(n)}_j\r)_{j \geq 1} \deq \l(\w^{(n)}_j\r)_{j \geq 1}$, such that $\hat{\bxi}^{(n)}_j \to \hat{\bxi}_j\sim \mu_0$, almost surely, as $n \to \infty$, independently for $j \geq 1$. Define $\hat{\bmu}^{(n)} = \sum_{j \geq 1}\hat{\w}_j^{(n)}\delta_{\hat{\bxi}^{(n)}_j} + \hat{\w}_0^{(n)}\mu^{(n)}_0$, where $\hat{\w}_0^{(n)} = 1-\sum_{j \geq 1}\hat{\w}^{(n)}_j$. Then for any continuous and bounded function, $f$, by  Lemma \ref{lem:Emu(f)}
\begin{equation}\label{eq:muL2_1}
\begin{aligned}
\Esp&\l[\l\{\hat{\bmu}^{(n)}(f)-\delta_{\hat{\bxi}_1}(f)\r\}^2\r]\\
& = \Esp\l[\l\{\hat{\bmu}^{(n)}(f)\r\}^2\r]- 2\Esp\l[\hat{\bmu}^{(n)}(f)\,f\l(\hat{\bxi}_1\r)\r] + \Esp\l[\l\{f\l(\hat{\bxi}_1\r)\r\}^2\r]\\
& =  \rho^{(n)}\, \mu^{(n)}_0\l(f^2\r) + \l(1-\rho^{(n)}\r)\l\{\mu^{(n)}_0(f)\r\}^2 - 2\Esp\l[\hat{\bmu}^{(n)}(f)\,f\l(\hat{\bxi}_1\r)\r] + \mu_0(f^2).
\end{aligned}
\end{equation}
As $f$ is bounded, we can write
\begin{equation*}
\begin{aligned}
\Esp\l[\hat{\w}^{(n)}_1\r]\Esp\l[f\l(\hat{\bxi}^{(n)}_1\r)f\l(\hat{\bxi}_1\r)\r]-&M^2\sum_{j \neq 1}\Esp\l[\hat{\w}^{(n)}_j\r]  \leq \Esp\l[\hat{\bmu}^{(n)}(f)\,f(\hat{\bxi}_1)\r]\\
& \leq \Esp\l[\hat{\w}^{(n)}_1\r]\Esp\l[f\l(\hat{\bxi}^{(n)}_1\r)f\l(\hat{\bxi}_1\r)\r]+M^2\sum_{j \neq 1}\Esp\l[\hat{\w}^{(n)}_j\r]
\end{aligned}
\end{equation*}
where $M$ is a bound of $f$. By taking limits as $n \to \infty$ in the last equation and by \eqref{eq:Ew1Ew0}, we get
\[
\Esp\l[\hat{\bmu}^{(n)}(f)\,f\l(\hat{\bxi}_1\r)\r] \to \Esp\l[\l\{f\l(\hat{\bxi}_1\r)\r\}^2\r] = \mu_0(f^2).
\]
Hence, by making $n \to \infty$ in \eqref{eq:muL2_1}, we obtain
\[
\Esp\l[\l\{\hat{\bmu}^{(n)}(f)-\delta_{\hat{\bxi}_1}(f)\r\}^2\r] \to 0.
\]
That is $\hat{\bmu}^{(n)}(f) \to \delta_{\hat{\bxi}_1}(f)$ in $\cl{L}_2$, which implies $\bmu^{(n)}(f) \deq \hat{\bmu}^{(n)}(f) \dto \delta_{\hat{\bxi}_1}(f)$. As this holds for every continuous and bounded function, $f$, we obtain $\bmu^{(n)} \dwto \delta_{\hat{\bxi}_1}$, as $n \to \infty$. Set $\hat{\x} = \hat{\bxi}_1$, under analogous arguments as in (i), we may assume without loss of generality that $\hat{\bmu}^{(n)}$ converges weakly almost surely to $\delta_{\hat{\x}}$ as $n \to \infty$, and consider exchangeable sequences $\l\{\hat{\X}^{(n)} = \l(\hat{\x}^{(n)}_i\r)_{i \geq 1}\r\}_{n \geq 1}$, such that $\hat{\X}^{(n)}$ is directed by $\hat{\bmu}^{(n)}$. Fix $m \geq 1$ and $B_1,\ldots,B_m \in \B_S$. The diffuseness of $\mu_0$ implies that $\hat{\x} \not\in \partial B_i$ almost surely, so that outside a $\hat{\Prob}$-null set, $\delta_{\hat{\x}}(\partial B_i) = 0$, and using the Portmanteau theorem we obtain $\bmu^{(n)}(B_i) \to \delta_{\hat{\x}}(B_i)$, almost surely, as $n \to \infty$. The representation theorem for exchangeable sequences assures
\[
\hat{\Prob}\l[\bigcap_{i=1}^m \l(\hat{\x}^{(n)}_i \in B_i\r)\mi \hat{\bmu}^{(n)}\r] = \prod_{i=1}^{m}\hat{\bmu}^{(n)}(B_i) \to \prod_{i=1}^{m}\delta_{\hat{\x}}(B_i) ,
\]
almost surely, as $n \to \infty$, and by taking expectations we get
\[
\hat{\Prob}\l[\bigcap_{i=1}^m \l(\hat{\x}^{(n)}_i \in B_i\r)\r] \to  \hat{\Esp}\l[\prod_{i=1}^{m}\delta_{\hat{\x}}(B_i)\r]  = \hat{\Prob}\l[\hat{\x} \in B_1,\ldots,\hat{\x} \in B_m\r].
\]
Hence $\X^{(n)} \deq \hat{\X}^{(n)} \dto (\hat{\x},\hat{\x},\ldots)$ as $n \to \infty$. 
\end{proof}

The proof of Theorem \ref{theo:rho_limit} appears in \cite{GvdV17} for the particular case of the  Dirichlet process. Note that the elements of any exchangeable sequence are marginally  identically distributed, so in terms of their mutual dependence the two extrema are the case where the random variables are i.i.d. and the case where they are identical. This are precisely the two limits of exchangeable sequences driven by a species sampling process, when the tie probability approaches zero or one, respectively.

\section{Proof of main results}\label{sec:main-proofs}

This section is dedicated to prove the main results of the article.

\subsection{Proof of Theorem \ref{theo:exch_SB}}

\begin{proof}
(i) Following the proof of Proposition 7 in \cite{BO14}, it suffices to show that for every $0 < \varepsilon' < 1$, there exist $0 < \delta <\varepsilon'$ such that $\Prob\l[\bigcap_{i=1}^n (\delta < \v_i < \varepsilon')\r] > 0$, for every $n \geq 1$. Fix $0 < \varepsilon' < 1$ and consider $\varepsilon'' = \min\{\varepsilon,\varepsilon'\}$, where $\varepsilon > 0$ is such $(0,\varepsilon)$ is contained in the support of $\nu_0$. Set $\delta = \varepsilon''/2$, by the representation theorem for exchangeable sequences, Jensen's inequality and the fact that $(\delta,\varepsilon'') \subseteq (0,\varepsilon)$ is contained in the support of $\nu_0$,
\[
\Prob\l[\bigcap_{i=1}^n (\delta < \v_i < \varepsilon'')\r] = \Esp\l[\prod_{i=1}^n\bnu(\delta,\varepsilon'')\r] = \Esp\l[\l\{\bnu(\delta,\varepsilon'')\r\}^n\r] \geq \{\nu_0(\delta,\varepsilon'')\}^n > 0,
\]
for every $n \geq 1$. As $\varepsilon'' \leq \varepsilon'$, we conclude $\Prob\l[\bigcap_{i=1}^n (\delta < \v_i < \varepsilon')\r] \geq \Prob\l[\bigcap_{i=1}^n (\delta < \v_i < \varepsilon'')\r] >0$, for $n \geq 1$.

(ii) As explained in \cite{GvdV17}, $1-\sum_{i=1}^j \w_i = \prod_{i=1}^{j}(1-\v_i)$, for every $j \geq 1$. From which is evident that, $\sum_{j \geq 1}\w_j = 1$ if and only if $\prod_{i=1}^{j}(1-\v_i) \to 0$, as $j \to \infty$, almost surely. Since $0 \leq \prod_{i=1}^{j}(1-\v_i) \leq 1$, this is  equivalent to $\Esp\l[\prod_{i=1}^{j}(1-\v_i)\r] \to 0$. As $(\v_i)_{i \geq 1}$ is exchangeable and directed by $\bnu$, we get $\Esp\l[\prod_{i=1}^{j}(1-\v_i)\r] = \Esp\l[(1-\Esp[\v_1\mid\bnu])^j\r]$. Now, if $\bnu(\{0\}) < 1$ almost surely, then $\Prob\l[\v_1 > 0\mid \bnu\r] > 0$, almost surely. Since $\v_1$ is non-negative, this shows $\Esp[\v_1\mid\bnu] > 0$ almost surely, hence $\Esp\l[\prod_{i=1}^{j}(1-\v_i)\r] = \Esp\l[(1-\Esp[\v_1\mid\bnu])^j\r] \to 0$ as $j \to \infty$. Alternatively, if $\Prob\l[\bnu(\{0\}) = 1\r] > 0$, then for every $j \geq 1$,
\[
\Esp\l[(1-\Esp[\v_1\mid\bnu])^j\r] = \Esp\l[(1-\Esp[\v_1\mid\bnu])^j\Ind_{\{\bnu(\{0\}) < 1\}}\r] + \Prob\l[\bnu(\{0\}) = 1\r].
\]
Which implies 
\[
\lim_{j \to \infty}\Esp\l[\prod_{i=1}^{j}(1-\v_i)\r] = \lim_{j \to \infty}\Esp\l[(1-\Esp[\v_1\mid\bnu])^j\r] \geq \Prob\l[\bnu(\{0\}) = 1\r] > 0.
\]
\end{proof}

\subsection{Proof of Theorem \ref{theo:exch_SB_limit_0}}

\begin{proof}
(i) First we see that the sequences of length variables $\V^{(n)} = \l(\v^{(n)}_i\r)_{i \geq 1}$ converge in distribution to $\V = (\v_i)_{i \geq 1}$. Since $[0,1]$ and $\cl{P}([0,1])$ are Polish, we might construct on some probability space $\l(\hat{\Omega},\hat{\F},\hat{\Prob}\r)$ some exchangeable sequences $\l\{\hat{\V}^{(n)} = \l(\hat{\v}^{(n)}_i\r)_{i \geq 1}\r\}_{n \geq 1}$, such that $\hat{\V}^{(n)}$ is directed by a random probability measure $\hat{\bnu}^{(n)} \deq \bnu^{(n)}$, and where $\hat{\bnu}^{(n)}$ converges weakly almost surely to $\nu_0 \neq \delta_0$, as $n \to \infty$. Fix $m \geq 1$ and $B_1,\ldots,B_m \in \B_{[0,1]}$ such that $\nu_0(\partial B_i) = 0$, for every $i \leq m$. By the Portmanteau theorem we know $\hat{\bnu}(B_i) \to \nu_0(B_i)$ almost surely as $n \to \infty$. This together with the representation theorem for exchangeable sequences imply
\[
\hat{\Prob}\l[\bigcap_{i=1}^m \l(\hat{\v}^{(n)}_i \in B_i\r)\mi \hat{\bnu}^{(n)}\r] = \prod_{i=1}^{m}\hat{\bnu}^{(n)}(B_i) \to \prod_{i=1}^{m}\nu_0(B_i),
\]
almost surely, as $n \to \infty$, and by taking expectations we obtain
\[
\hat{\Prob}\l[\bigcap_{i=1}^m \l(\hat{\v}^{(n)}_i \in B_i\r)\r] \to \prod_{i=1}^{m}\nu_0(B_i) = \Prob\l[\bigcap_{i=1}^m \l(\v_i \in B_i\r)\r].
\]
Since each sequence of length variables is countable it is enough to prove the convergence of the finite dimensional distributions, and we get $\V^{(n)} \deq \hat{\V}^{(n)} \dto \V$ as desired. Note that mapping
\[
(v_1,v_2,\ldots,v_j) \mapsto \l(v_1,v_2(1-v_1),\ldots,v_j\prod_{i=1}^{j-1}(1-v_i)\r)
\]
is continuous with respect to the product topology, thus the weights of $\bmu^{(n)}$, $\W^{(n)} = \SB\l[\V^{(n)}\r]$, converge in distribution to the weights of $\bmu$, $\W = \SB\l[\V\r]$. Further, the requirements on $\bnu^{(n)}$ and $\nu_0$ assure $\W^{(n)}$ and $\W$ take values in the infinite dimensional simplex, $\Delta_{\infty}$. In addition, as the base measures $\mu^{(n)}_0$ converge weakly to $\mu_0$, we also have that the atoms of $\bmu^{(n)}$, $\bXi^{(n)} = \l(\bxi^{(n)}_j\r)_{j \geq 1}$ (which are independent of $\W^{(n)}$) converge in distribution to the atoms of $\bmu$, $\bXi = (\bxi_j)_{j \geq 1}$ (which are independent of $\W$). Thus $\l(\bXi^{(n)},\W^{(n)}\r) \dto (\bXi,\W)$ in $S^{\infty} \times \Delta_{\infty}$, and Lemma \ref{lem:dL_cont_map} yields $\bmu^{(n)}$ converges weakly in distribution to $\bmu$. In particular, if $\nu_0$ denotes a $\Be(1,\theta)$ distribution we get $\bmu$ is a Dirichlet process, and as shown by \cite{P96b}, $\W$ is in size-biased order.

(ii) Analogously as in (i) we may construct sequences $\l\{\hat{\V}^{(n)} = \l(\hat{\v}^{(n)}_i\r)_{i \geq 1}\r\}_{n \geq 1}$, such that $\hat{\V}^{(n)}$ is directed by a random probability measure $\hat{\bnu}^{(n)} \deq \bnu^{(n)}$, and where $\hat{\bnu}^{(n)}$ converges weakly almost surely to $\delta_{\hat{\v}}$, with $\hat{\v} \sim \nu_0$, as $n \to \infty$. Fix $m \geq 1$ and $B_1,\ldots,B_m \in \B_{[0,1]}$ with $\nu_0(\partial B_i) = 0$, for every $i \leq m$. Then we get that $\hat{\v} \not\in \partial B_i$ almost surely, so that outside a $\hat{\Prob}$-null set, $\delta_{\hat{\v}}(B_i) = 0$, and using the Portmanteau theorem we obtain $\bnu^{(n)}(B_i) \to \delta_{\hat{\v}}(B_i)$, almost surely, as $n \to \infty$. The representation theorem for exchangeable sequences assures
\[
\hat{\Prob}\l[\bigcap_{i=1}^m \l(\hat{\v}^{(n)}_i \in B_i\r)\mi \hat{\bnu}^{(n)}\r] = \prod_{i=1}^{m}\hat{\bnu}^{(n)}(B_i) \to \prod_{i=1}^{m}\delta_{\hat{\v}}(B_i) ,
\]
almost surely, as $n \to \infty$, and by taking expectations we get
\[
\hat{\Prob}\l[\bigcap_{i=1}^m \l(\hat{\v}^{(n)}_i \in B_i\r)\r] \to  \hat{\Esp}\l[\prod_{i=1}^{m}\delta_{\hat{\v}}(B_i)\r]  = \hat{\Prob}\l[\hat{\v} \in B_1,\ldots,\hat{\v} \in B_m\r].
\]
Hence $\V^{(n)} \deq \hat{\V}^{(n)} \dto (\hat{\v},\hat{\v},\ldots) \deq (\v,\v,\ldots)$ as $n \to \infty$. Note that in this case, $\W = \SB[(\v,\v,\ldots)]$ is in decreasing order. The rest of the proof of (ii) follows identically is in (i).
\end{proof}

\subsection{Proof of Corollary \ref{cor:exch_SB_limit}}

\begin{proof}
By Theorem \ref{theo:rho_limit} we know that if $\rho_{\nu}^{(n)} \to 0$ then, $\bnu^{(n)}$ converges weakly in distribution to $\nu_0$. Alternatively, if $\rho_{\nu}^{(n)} \to 1$ we get $\bnu^{(n)}$ converges weakly in distribution to $\delta_{\v}$, where $\v \sim \nu_0$. The result then follows from Theorem \ref{theo:exch_SB_limit_0}.
\end{proof}

\subsection{Proof of Theorem \ref{theo:weights_ord}}

\begin{proof}
Fix $j \geq 1$. As $\nu_0$ diffuse, $\l(1-\v_j\r) >0$ almost surely, for every $\rho_{\nu}\in(0,1)$. Hence, $\w_j \geq \w_{j+1}$ if and only if $\v_j \geq \v_{j+1}(1-\v_j)$, or equivalently $\v_{j+1} \leq c\l(\v_j\r)$ where $c(v) = 1\wedge v(1-v)^{-1}$. Using the exchangeability of $\l(\v_i\r)_{i \geq 1}$, we know that under the event $\l\{\v_j \neq \v_{j+1}\r\}$, which occurs with probability $1-\rho_{\nu}$, the conditional distribution of $\l(\v_j,\v_{j+1}\r)$ is that of $(\v^*,\v)$, where $\v^*$ and $\v$ are i.i.d. from $\nu_0$ (see for instance Theorem \ref{theo:Rep_SSP}). Hence we can easily compute
%\begin{equation}
\begin{align}
\Prob\l[\w_j \geq \w_{j+1}\r] & =\Prob\l[\v_{j+1} \leq c\l(\v_j\r)\mi \v_j = \v_{j+1}\r]\rho_{\nu} + \label{eq:prob_w_rho}\\
& \quad \quad \quad \quad \quad \quad \quad \quad \Prob\l[\v_{j+1} \leq c\l(\v_j\r)\mi \v_j \neq \v_{j+1}\r](1-\rho_{\nu})\nonumber\\
& = \rho_{\nu} +(1-\rho_{\nu})\Prob\l[\v^* \leq c(\v)\r].\nonumber
\end{align}
%\end{equation}
Noting that $\Prob\l[\v^* \leq c(\v)\r] = \Esp\l[\overrightarrow{\nu_0}(c(\v))\r]$, where $\overrightarrow{\nu_0}$ is the distribution function of $\v \sim \nu_0$, we get (a). To prove (b) first we show that
\begin{equation}\label{eq:w_dec_v1...vj}
\Prob[\w_{j}\geq \w_{j+1} \mid \w_1,\ldots,\w_{j}] = \Prob[\v_{j+1}\leq c(\v_j) \mid \v_1,\ldots,\v_{j}],
\end{equation}
almost surely. It is straight forward from the stick-breaking construction that $(\w_1,\ldots,\w_j)$ is $(\v_1,\ldots,\v_j)$-measurable. Conversely, the proof of Theorem \ref{theo:sb} yields $(\v_1,\ldots,\v_j)$ is $(\w_1,\ldots,\w_j)$-measurable as well, whenever $0 < \v_i < 1$ almost surely for every $i \geq 1$, this is of course the case of ESBs, because the underlying base measure, $\nu_0$, is diffuse. Finally as explained above $\w_j \geq \w_{j+1}$ if and only if $c(\v_j) \geq \v_{j+1}$, which proves \eqref{eq:w_dec_v1...vj}. Now, since $(\v_i)_{i \geq 1}$ is sampled from a species sampling process, we already know, from Theorem \ref{theo:Rep_SSP}, how to compute 
\[
\Prob[\v_{j+1}\leq c(\v_j) \mid \v_1,\ldots,\v_{j}] = \sum_{i=1}^{\K_j}\frac{\pi_{\nu}\l(\n^{(i)}\r)}{\pi_{\nu}(\n)}\Ind_{\{\v^*_i \leq c(\v_j)\}} + \frac{\pi_{\nu}\l(\n^{(\K_j+1)}\r)}{\pi_{\nu}(\n)}\overrightarrow{\nu_0}(c(\v_j)),
\] 
which finishes the proof.
\end{proof}

\subsection{Proof of Corollary \ref{cor:weights_ord}}

\begin{proof}
For $\v \sim \Be(1,\theta)$, its distribution function is given by, $\overrightarrow{\nu_0}(x) = 1-(1-x)^{\theta}$, hence by substituting the tie probability $\rho_{\nu} = 1/(\beta +1)$, in Theorem \ref{theo:weights_ord}, we obtain 
\begin{equation}\label{eq:w(a)}
\Prob\l[\w_j \geq \w_{j+1}\r] = 1 - \frac{\beta}{\beta +1}\Esp\l[(1-c(\v))^{\theta}\r].
\end{equation}
where $c(v) = 1 \wedge v(1-v)^{-1}$. Since  $\v \sim \Be(1,\theta)$, we get
\begin{align*}
\Esp\l[(1-c(\v))^{\theta}\r] = \theta\int_0^{1/2}\l(1-\frac{x}{1-x}\r)^{\theta}(1-x)^{\theta-1}dx = \theta\int_0^{1/2}\frac{(1-2x)^{\theta}}{(1-x)}dx,
\end{align*}
and by the change of variables $y = 2x$, 
\begin{align*}
\Esp\l[(1-c(\v))^{\theta}\r] = \frac{\theta}{2} \int_0^1 \frac{(1-y)^\theta}{(1-y/2)}dy = \frac{\2F1(1,1;\theta+2,1/2)\theta}{2(\theta+1)}.
\end{align*}
Substituting this quantity into \eqref{eq:w(a)} yields (a). For the proof of (b) we simply have to recall that if $(\v_i)_{i \geq 1}$ are exchangeable and driven by a Dirichlet process, $\bnu$, with total mass parameter $\beta$ and base measure $\nu_0 =\Be(1,\theta)$, then
\[
\Prob[\v_{j+1}\in \cdot\mid \v_1,\ldots,\v_j] = \frac{1}{\beta+j}\l\{\sum_{i = 1}^{\K_j}\n_i\delta_{\v_i^*} + \beta\nu_0\r\},
\]
where $\v^*_1,\ldots,\v^{*}_{\K_j}$ are the distinct values that $\{\v_1,\ldots,\v_j\}$ exhibits and $\n_{i} = |\{l\leq j:\v_l = \v^{*}_i\}|$, for every $i \leq \K_j$. This implies
\[
\Prob[\v_{j+1}\leq c(\v_j)\mid \v_1,\ldots,\v_j] = \frac{1}{\beta+j}\l\{\sum_{i=1}^{\K_j}\n_i\Ind_{\{\v^*_i \leq c(\v_j)\}} + \beta\overrightarrow{\nu_0}(c(\v_j))\r\},
\] 
and recalling that $\overrightarrow{\nu_0}(x) = 1-(1-x)^{\theta}$ the proof of (b) follows.
\end{proof}

\section{MCMC implementation for density estimation}\label{sec:MaP_EaP}
Say we model elements in $\{\y_1,\ldots,\y_n\}$ as i.i.d. sampled from the ESB mixture, $\bPhi = \sum_{j \geq 1}\w_jG(\cdot\mid\bxi_j)$  where $G(\cdot\mid s)$ has a density, denoted by the same letter, with respect to suitable measure, and $\bmu = \sum_{j \geq 1}\w_j \delta_{\bxi_j}$ defines an ESB with exchangeable length variables $\V = (\v_i)_{i\geq 1}$ driven by a species sampling process $\bnu$. Let us denote $\W = (\w_j)_{j \geq 1} = \SB[\V]$ and $\bXi = (\bxi_j)_{j \geq 1}$, we will also use the notation $\bbp(\z)$ and $\bbp(\z\mid \bgamma)$ to refer to the marginal density (or mass probability function) of $\z$, and the conditional density of $\z$ given $\bgamma$, respectively.

Consider the random density
\begin{equation*}\label{eq:f(y)_int}
\bPhi(\y) = \bbp(\y\mid \W,\bXi) = \sum_{j \geq 1}\w_jG(\y|\bxi_j),
\end{equation*}
for MCMC implementation purposes, and following \cite{W07}, this random density can be augmented as
\begin{equation*}\label{eq:f(y,u)}
\bbp(\y,\u|\W,\bXi) = \sum_{j \geq 1} \Ind_{\{\u < \w_j\}}G(\y|\bxi_j),
\end{equation*} 
Given $\u$, the number of components in the mixture is finite, with indexes being the elements of $A_\u(\W) = \{j: \u < \w_j\}$, that is 
\begin{equation}\label{eq:f(y|u,w,s)}
\bbp(\y|\u,\W,\bXi) = \frac{1}{|A_\u(\W)|} \sum_{j \in A_\u(\W)} G(\y|\bxi_j).
\end{equation}
This translates the problem of dealing with a mixture that has a infinite number of components, into working with one that features a random number of components. From a numerical perspective, the latter is much more manageable. Moreover, considering the latent allocation variable $\d$, i.e. $\d = j$ iff $\y$ is sampled from $G(\cdot|\bxi_j)$, one can further consider the augmented joint density,
\begin{equation}\label{eq:f(y,u,d)}
\bbp(\y,\u,\d|\W,\bXi) = \Ind_{\{\u<\w_d\}}G(\y|\bxi_d).
\end{equation}
The complete data likelihood based on a sample of size $n$ from \eqref{eq:f(y,u,d)} is easily seen to be
\[
\mathL_{\bXi,\W}((\y_k,\u_k,\d_k)_{k=1}^{n}) = \prod_{k=1}^{n} \Ind_{\{\u_k < \w_{d_k}\}}G(\y_k|\bxi_{\d_k}),
\]
and taking this into account we can compute the full conditional distributions, required to perform the Gibbs sampler, as described below.

\subsubsection*{1. Updating the slice variables, $\U = (\u_k)_{k=1}^{n}$:}

\begin{equation*}\label{eq:upd_u_k}
\bbp(\u_k|\ldots) \propto \Ind_{\{\u_k < \w_j\}},
\end{equation*}
Hence to updated $\u_k$, we simply sample it from a $\msf{Unif}\l(0,\w_k\r)$ distribution.

\subsubsection*{2. Updating the kernel parameters, $\bXi=(\bxi_j)_{j\geq 1}$:}
Under the assumption that the base measure, $\mu_0$ has a density, denoted by the same symbol, with respect to a suitable measure, we get
\begin{equation*}\label{eq:upd_theta_j}
\bbp(\bxi_j|\ldots) \propto \mu_0(\bxi_j)\prod_{k \in D_j}G(\y_k|\bxi_j),
\end{equation*}
where $D_j = \{k: \d_k = j\}$. If $\mu_0$ and $G$ form a conjugate pair, the above is easy to sample from.

\subsubsection*{3. Updating latent allocation variables, $\D = (\d_k)_{k=1}^{n}$:}

\begin{equation*}\label{eq:upd_d_k}
\bbp(\d_k = j|\ldots) \propto G(\y_k|\bxi_j)\Ind_{\{\u_k < \w_j\}}, 
\end{equation*}
which is a discrete distribution with finite support, hence easy to sample from. 

\subsubsection*{4. Updating the length variables, $\V = (\v_i)_{i \geq 1}$:}
Consistently with the notation of previous sections, let $\pi_{\nu}$ denote the EPPF corresponding to $\bnu$. Let $\V_{-j}$, denote the vector of (previously updated) $\v_i$'s excluding $\v_j$, by Remark \ref{rem:varphi} below, $\V_{-j}$ is finite  at each iteration. Hence exploiting the exchangeability of $\V$ we get
\begin{equation}\label{eq:pred_v_prior}
\bbp(\v_j\mid\V_{-j}) = \sum_{i=1}^{\k}\frac{\pi_{\nu}(\n^{(i)})}{\pi_{\nu}(\n)}\delta_{\v^*_i}(\v_j) + \frac{\pi_{\nu}(\n^{(\k+1)})}{\pi_{\nu}(\n)}\nu_0(\v_j),
\end{equation}
where, $\v^*_1,\ldots,\v^*_{\k}$, are the distinct values $\V_{-j}$ exhibits, $\n = (\n_1,\ldots\n_{\k})$,  $\n^{(i)} = (\n_1,\ldots \n_{i-1},\n_i+1,\n_{i+1}, \ldots,\n_{\k})$ and  $\n^{(\k+1)} = (\n_1,\ldots,\n_{\k},1)$,  with $\n_{l} = |\{\v_i \in  \V_{-j}:\v_i = \v^{*}_l\}|$. This yields the full conditional of $\v_j$ is
\begin{align*}
\bbp(\v_j|\ldots) & \propto \left[\prod_{k=1}^{m}\Ind_{\{\u_k<\w_{\d_k}\}}\right]\l[ \sum_{i=1}^{\k}\pi_{\nu}(\n^{(i)})\delta_{\v^*_i}(\v_j) + \pi_{\nu}(\n^{(\k+1)})\nu_0(\v_j)\r].
\end{align*}
After some algebra and recalling $\w_{\d_k} = \v_{\d_k}\prod_{i=1}^{\d_k -1}(1-\v_i)$, it can be seen that
\[
\prod_{k=1}^{m}\Ind_{\{\u_k<\w_{\d_k}\}} \propto \Ind_{\{a_j < \v_j < b_j\}}
\]
where the proportionality sign is with respect to $\v_j$ and where
\[
a_j = \max_{k \in A_j} \l\{\frac{\u_k}{\prod_{i<\d_k}(1-\v_i)}\r\}\quad
\text{ and } \quad
b_j = 1-\max_{k \in B_j}\l\{\frac{\u_k}{\v_{\d_k}\prod_{i<\d_k,i\neq j}(1-\v_i)}\r\},
\]
with $A_j = \{k:\d_k = j\}$ and $B_j = \{k:\d_k > j\}$, and using the convention that $a_j=0$ if $A_j = \emptyset$ and $b_j = 1$ when $B_j = \emptyset$. Thus,
\begin{equation}\label{eq:pred_v_post}
\bbp(\v_j|\ldots) \propto \sum_{i=1}^{\k}\pi_{\nu}(\n^{(i)})\delta_{\v^*_i}(\v_j)\Ind_{\{a_j < \v^*_i < b_j\}} + \pi_{\nu}(\n^{(\k+1)})\nu_0(\v_j)\Ind_{\{a_j < \v_j < b_j\}}.
\end{equation}
Sampling $\v_j$ from \eqref{eq:pred_v_post} simply means that with probability $\mbf{q}_i \propto \pi_{\nu}\l(\n^{(i)}\r)\Ind_{\{a_j < \v^*_i < b_j\}}$ we set $\v_j = \v^{*}_i$ for every $1 \leq i \leq \k$, or with probability 
\[
\mbf{q}_0 \propto \pi_{\nu}\l(\n^{(\k+1)}\r)\int_{a_j}^{b_j}\nu_0(x) dx
\]
we sample $\v_j$ from the density 
\[ 
\frac{\nu_0(v)\Ind_{\{a_j < v < b_j\}}}{\int_{a_j}^{b_j}\nu_0(x) dx}.
\]
For example, if the mixing prior corresponds to a DSB, $\bmu = \sum_{j\geq 1}\w_j\delta_{\bxi_j}$, with parameters $(\beta,\theta,\mu_0)$, we see that \eqref{eq:pred_v_prior} becomes
\[
\bbp(\v_j\mid\V_{-j}) = \sum_{i=1}^{\k}\frac{\n_i}{|\V_{-j}|+\beta}\delta_{\v^*_i}(\v_j) + \frac{\beta}{|\V_{-j}|+\beta}\Be(\v_j\mid 1,\theta).
\]
Hence, we obtain
\[
\mbf{q}_i = \frac{\n_i\Ind_{\{a_j < \v^*_i < b_j\}}}{\sum_{l\in C_i}\n_l + \beta\l[(1-a_j)^{\theta} - (a-b_j)^{\theta}\r]},
\]
for every $1 \leq i \leq \k$, where $C_j = \{i\leq \K:a_j < \v^{*}_i < b_j\}$,
\[
\mbf{q}_0 = \frac{\beta\l[(1-a_j)^{\theta} - (a-b_j)^{\theta}\r]}{\sum_{l\in C_i}\n_l + \beta\l[(1-a_j)^{\theta} - (a-b_j)^{\theta}\r]},
\]
and 
\[
\frac{\nu_0(v)\Ind_{\{a_j < v < b_j\}}}{\int_{a_j}^{b_j}\nu_0(x) dx} = \frac{\theta (1-v)^{\theta-1}\Ind_{\{a_j < v < b_j\}}}{\l[(1-a_j)^{\theta} - (a-b_j)^{\theta}\r]}.
\]
Sampling from this last density is easy by means of inverse sampling.

\begin{rem}[For the updating of  $\bXi$, $\V$ and $\X$]\label{rem:varphi}
It is not necessary to sample $\v_j$ and $\bxi_j$ for infinitely many indexes, it suffices to sample them for $j \leq \varphi$,  where $\sum_{j=1}^{\varphi} \w_j \geq \max_k(1-\u_k)$, then it is not possible that $\w_j > \u_k$ for any $k \leq m$ and $j > \varphi$, and the updating of the latent allocation variables $\d_k$'s can take place.
\end{rem}

\subsubsection*{5. Updating the tie probability for DSBs:}

When implementing certain ESB priors such as DSBs, it is of interest to estimate the underlying tie probability $\rho_{\nu} = \Prob[\v_j = \v_l] = 1/(\beta+1)$, to do this we can assign it a prior distribution. Roughly speaking, this allows the model to choose between DSB mixtures that behave arbitrarily similar to a Dirichlet mixture, to a Geometric mixture or some other  mixture in between. It is straightforward to check that the full conditionals of $\bxi_j$, $\d_i$, $\u_i$ and $\v_j$ will remain as described above conditionally given $\beta = (1-\rho_{\nu})/\rho_{\nu}$. In this case we will require to update $\rho_{\nu}$ at each iteration of the Gibbs sampler. Since $\rho_{\nu}$ only affects directly the length variables, it is easy to see that the full conditional distribution of this random variable is
\[
\bbp(\rho_{\nu}\mid\ldots) \propto \bbp(\v_1,\v_2,\ldots\mid\rho_{\nu})\bbp(\rho_{\nu}).
\]
Given that at each iteration of the Gibbs sampler we only sample finitely many length variables, say $\v_1,\ldots,\v_m$, we obtain
\[
\bbp(\rho_{\nu}\mid\ldots) \propto \pi_{\nu}(\n_1,\ldots\n_{\K_m})\bbp(\rho_{\nu})
\]
(see Theorem \ref{theo:Rep_SSP}) where $\K_m$ is the number of distinct values $\{\v_1,\ldots,\v_m\}$ exhibits, $\n_1,\ldots,\n_{\K_m}$ are the frequencies of the distinct values and $\pi_{\nu}$ is the EPPF of the Dirichlet process. That is,
\[
\bbp(\rho_{\nu}\mid\ldots) \propto\frac{(1-\rho_{\nu})^{\K_m-1}\,\rho_{\nu}^{m-\K_m}}{\prod_{l=0}^{m-2}(1+l\rho_{\nu})}\bbp(\rho_{\nu}).
\]
Drawing samples from the full conditional of $\rho_{\nu}$ is possible with the aid of a rejection sampling method such as Adaptive Rejection Metropolis Sampling (ARMS) \citep[e.g.][]{GBT95}.

\subsection*{Posterior estimation}

Given the samples, $\l\{\l(\bxi^{(m)}_j\r)_j,\l(\w^{(m)}_j\r)_j,\l(\u^{(m)}_k\r)_k,\l(\d^{(m)}_k\r)_k\r\}_{m = 1}^{M}$, obtained after $M$ iterations of the Gibbs sampler after the burn-in period has elapsed, we can estimate the density of the data at $y$, by means of the expected a posteriori (EAP),
\begin{equation}\label{eq:hat_f}
\bPhi_{\mrm{EAP}}(y) = \Esp\l[\bPhi(y)\mi\y_1,\ldots\y_n\r] \approx \frac{1}{M}\sum_{m=1}^{M} \frac{1}{n}\sum_{k=1}^{n} \frac{1}{\big|A_k^{(m)}\big|} \sum_{j \in A_k^{(m)}} G\l(y\mi \bxi^{(m)}_j\r),
\end{equation}
where $A_k^{(m)} = \l\{j: \u_k^{(m)} < \w_j^{(m)}\r\}$. We can estimate as well the posterior distribution of $\K_n$ through
\begin{equation}\label{eq:post_Kn}
\Prob\l[\K_n = j\big|\y_1,\ldots,\y_n\r] \approx \frac{1}{M} \sum_{m=1}^{M} \Ind_{\l\{\K_n^{(m)} = j\r\}},
\end{equation}
where $\K_n^{(m)}$ is the number of distinct values in $\l(\d^{(m)}_k\r)_k$. Alternatively, we can find
\begin{align*}
\hat{m} & = \underset{1 < m \leq M}{\mrm{arg}\,\mrm{max}}\,\l\{\bbp\l[\l(\bxi^{(m)}_j\r)_j,\l(\w^{(m)}_j\r)_j,\l(\u^{(m)}_k\r)_k,\l(\d^{(m)}_k\r)_k\mi \y_1,\ldots,\y_n\r]\r\}\\
& = \underset{1 < m \leq M}{\mrm{arg}\,\mrm{max}}\,\l\{\bbp\l[\l(\y_k,\d_k^{(m)},\u_k^{(m)}\r)_k\mi \l(\bxi_j^{(m)},\w_j^{(m)}\r)_j\r]\bbp\l[\l(\bxi_j^{(m)},\w_j^{(m)}\r)_j\r]\r\}\\
& = \underset{1 < m \leq M}{\mrm{arg}\,\mrm{max}}\,\l\{\prod_{k=1}^n G\l(\y_k \mi \bxi^{(m)}_{\d_k^{(m)}}\r) \Ind\l\{\u_k^{(m)} <\w^{(m)}_{\d_k^{(m)}}\r\}\bbp\l[\l(\bxi_j^{(m)},\w_j^{(m)}\r)_j\r]\r\}.
\end{align*}
and use it to approximate the maximum a posterior (MAP) of the density at $y$, through
\[
\bPhi_{\mrm{MAP}}(y) \approx  \frac{1}{n}\sum_{k=1}^{n} \frac{1}{\big|A_k^{(\hat{m})}\big|} \sum_{j \in A_k^{(\hat{m})}} G\l(y\big|\bxi^{(\hat{m})}_j\r),.
\]
By means of the MAP we can also estimate  the clusters of the data points, $\{\y_1,\ldots,\y_n\}$, via the mixture components
\[
\l\{\frac{1}{n}\sum_{k=1}^n\l|A_k^{(\hat{m})}\r|^{-1}\Ind_{\l\{j \in A_k^{(\hat{m})}\r\}}G\l(\cdot\mi\bxi^{(\hat{m})}_j\r)\r\}_{j},
\]
by defining
\[
\mbf{c}_k = \underset{j}{\mrm{arg}\,\mrm{max}}\,\l\{\frac{1}{n}\sum_{k=1}^n\l|A_k^{(\hat{m})}\r|^{-1}\Ind_{\l\{j \in A_k^{(\hat{m})}\r\}}G\l(\y_k\mi\bxi^{(\hat{m})}_j\r)\r\},
\]
for every $k \in \{1,\ldots,n\}$, and putting $\y_i$ and $\y_k$ in the same cluster if and only if $\mbf{c}_i = \mbf{c}_k$.

\section{Supplemental graphs of Section \ref{sec:ill_rand}}\label{sec:paw3d}

\begin{figure}[H]
\centering
\includegraphics[scale=0.76]{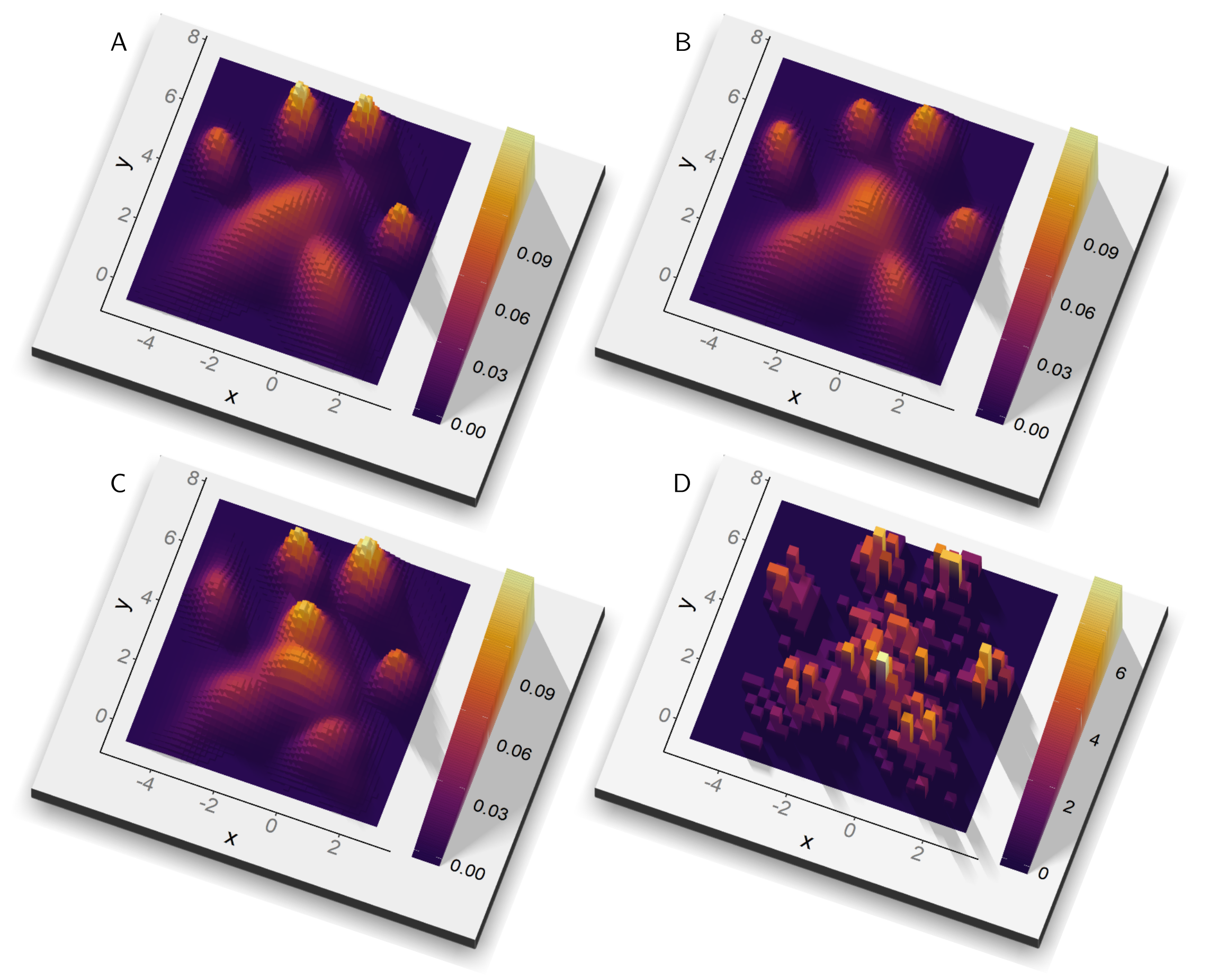} 
\caption{$3$D view of the estimated densities using the MAP, taking into account $8000$ iterations of the Gibbs sampler after a burn-in period of $2000$ iterations, according to the Dirichlet prior $(\msf{A})$ a DSB prior $(\msf{B})$  and a Geometric prior $(\msf{C})$. $\msf{D}$ shows the histogram of the data. \label{FigPaw3-3d-MaP}}
\end{figure}

\begin{figure}[H]
\centering
\includegraphics[scale=0.76]{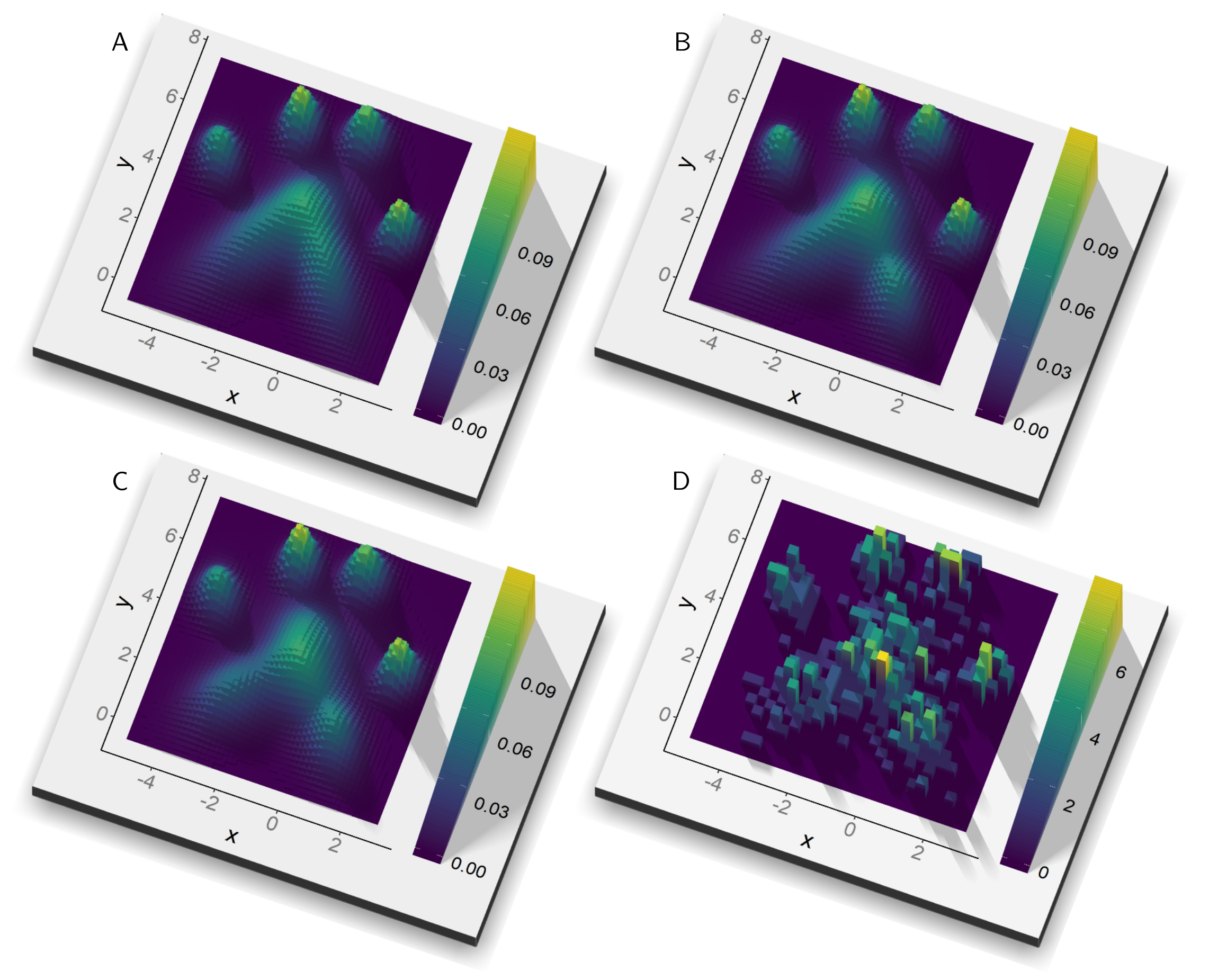} 
\caption{$3$D view of the estimated densities using the EAP, taking into account each fourth iteration among $8000$ iterations of the Gibbs sampler after a burn-in period of $2000$, according to the Dirichlet prior $(\msf{A})$ a DSB prior $(\msf{B})$ and a Geometric prior $(\msf{C})$. $\msf{D}$ shows the true density from which the data points were i.i.d. sampled. $\msf{D}$ shows the histogram of the data \label{FigPaw3-3d-EaP}}
\end{figure}

\bibliographystyle{dcu}
\bibliography{references}

\end{document}